\documentclass[11pt,a4paper,reqno,oneside]{amsart}
\usepackage[margin=2.8cm]{geometry}
\usepackage{amssymb}
\usepackage{graphicx} 
\usepackage[mathscr]{euscript}
\usepackage{enumerate}
\usepackage{xspace}
\usepackage{color}
\usepackage{enumerate}
\usepackage{enumitem}
\usepackage{amsthm}
\usepackage{dsfont}
\usepackage{bbm}
\usepackage[colorlinks=true, linkcolor = blue, citecolor = blue]{hyperref}
\usepackage[numbers]{natbib}
\usepackage[backgroundcolor=white,bordercolor=red]{todonotes}
\DeclareMathAlphabet{\mathpzc}{OT1}{pzc}{m}{it}
\usepackage[nameinlink]{cleveref}
\begin{document}

\theoremstyle{plain}

\newtheorem{theorem}{Theorem}[section]
\newtheorem{lemma}[theorem]{Lemma}
\newtheorem{proposition}[theorem]{Proposition}
\newtheorem{corollary}[theorem]{Corollary}
\newtheorem{definition}[theorem]{Definition}
\theoremstyle{definition}
\newtheorem{listing}[theorem]{Assumption}
\newtheorem{example}[theorem]{Example}
\newtheorem{remark}[theorem]{Remark}
\newtheorem{assumption}[theorem]{Assumption}
\newtheorem{SA}[theorem]{Standing Assumption}
\newtheorem{condition}[theorem]{Condition}

\renewcommand{\chapterautorefname}{Chapter} 
\renewcommand{\sectionautorefname}{Section}

\newcommand{\on}{\operatorname}
\crefname{theorem}{theorem}{theorems}
\Crefname{theorem}{Theorem}{Theorems}
\crefname{lemma}{lemma}{lemmas}
\Crefname{lemma}{Lemma}{Lemmata}
\crefname{proposition}{proposition}{propositions}
\Crefname{proposition}{Proposition}{Propositions}
\crefname{corollary}{corollary}{corollaries}
\Crefname{corollary}{Corollary}{Corollaries}
\crefname{example}{example}{examples}
\Crefname{example}{Example}{Examples}
\crefname{definition}{definition}{definitions}
\Crefname{definition}{Definition}{Definitions}
\crefname{remark}{remark}{remarks}
\Crefname{remark}{Remark}{Remarks}
\Crefname{listing}{Assumption}{Assumptions}
\crefname{condition}{condition}{conditions}
\Crefname{condition}{Condition}{Conditions}

\newcommand{\Y}{\mathfrak{X}}
\renewcommand{\k}{\kappa}
\newcommand{\bC}{\mathbf{C}}
\newcommand{\loc}{{\mathrm{loc}}}
\newcommand{\of}{[\hspace{-0.06cm}[}
\newcommand{\gs}{]\hspace{-0.06cm}]}
\newcommand{\A}{\mathbf{A}}
\newcommand{\B}{\mathbb{B}}
\newcommand{\p}{\mathfrak{p}}
\newcommand{\Q}{\mathds{Q}}
\newcommand{\q}{\mathfrak{q}}
\newcommand{\U}{\mathfrak{U}}
\renewcommand{\u}{\mathfrak{u}}
\let\SETMINUS\setminus
\renewcommand{\setminus}{\backslash}
\newcommand{\s}{\mathfrak{s}}
\newcommand{\m}{\mathfrak{m}}
\def\stackrelboth#1#2#3{\mathrel{\mathop{#2}\limits^{#1}_{#3}}}
\newcommand{\bF}{\mathbf{F}}
\newcommand{\bcF}{\mathbf{F}^\circ}
\newcommand{\cM}{\mathcal{M}}
\newcommand{\z}{\mathsf{Z}}
\renewcommand{\U}{\mathsf{U}}

\newcommand\myrot[1]{\mathrel{\rotatebox[origin=c]{#1}{$\Longrightarrow$}}}
% create four new angled double-struck arrows
\newcommand\NEarrow{\myrot{45}}
\newcommand\SEarrow{\myrot{-45}}

\renewcommand{\theequation}{\thesection.\arabic{equation}}
\numberwithin{equation}{section}
\newcommand{\X}{\mathsf{X}}
\newcommand{\Z}{\mathcal{Z}^\circ}
\newcommand{\disc}{\operatorname{disc}}

\renewcommand{\theequation}{\thesection.\arabic{equation}}
\numberwithin{equation}{section}

\newcommand\llambda{{\mathchoice
		{\lambda\mkern-4.5mu{\raisebox{.4ex}{\scriptsize$\backslash$}}}
		{\lambda\mkern-4.83mu{\raisebox{.4ex}{\scriptsize$\backslash$}}}
		{\lambda\mkern-4.5mu{\raisebox{.2ex}{\footnotesize$\scriptscriptstyle\backslash$}}}
		{\lambda\mkern-5.0mu{\raisebox{.2ex}{\tiny$\scriptscriptstyle\backslash$}}}}}

\newcommand{\tr}{\operatorname{tr}}
\newcommand{\dd}{d}
\newcommand{\cadlag}{c\`adl\`ag }

\newcommand{\F}{\mathbf{F}}
\newcommand{\1}{\mathds{1}}
\newcommand{\la}{\langle}
\newcommand{\ra}{\rangle}
\newcommand{\lle}{\langle\hspace{-0.085cm}\langle}
\newcommand{\rre}{\rangle\hspace{-0.085cm}\rangle}
\renewcommand{\epsilon}{\varepsilon}

\newcommand{\cF}{\mathcal{F}}
\newcommand{\ccF}{\mathcal{F}^\circ}
\newcommand{\Jto}{\to_{J_1}}
\newcommand{\Mto}{\to_{M_1}}
\newcommand{\D}{\mathbb{D}}
\newcommand{\C}{\mathbb{C}}
\newcommand{\N}{N}

\makeatletter
\@namedef{subjclassname@2020}{%
	\textup{2020} Mathematics Subject Classification}
\makeatother

\title[]{The Martingale Problem Method Revisited} %\\
\author[D. Criens]{David Criens}
\author[P. Pfaffelhuber]{Peter Pfaffelhuber}
\author[T. Schmidt]{Thorsten Schmidt}
\thanks{DC acknowledges financial support from the DFG project No. SCHM 2160/15-1.}
\address{Albert-Ludwigs University of Freiburg, Ernst-Zermelo-Str. 1, 79104 Freiburg, Germany}
\email{david.criens@stochastik.uni-freiburg.de}
\email{peter.pfaffelhuber@stochastik.uni-freiburg.de}
\email{thorsten.schmidt@stochastik.uni-freiburg.de}

\keywords{\vspace{1ex}Martingale problem, path space, fixed times of discontinuity, limit theorems, semimartingales, weak-strong convergence, stable convergence, Skorokhod topology, local uniform topology, Volterra equations}


\date{\today}
\maketitle

\frenchspacing
\pagestyle{myheadings}

\begin{abstract}
We use the abstract method of (local) martingale problems in order to give criteria for convergence of stochastic processes. Extending previous notions, the formulation we use is neither restricted to Markov processes (or semimartingales), nor to continuous or \cadlag paths. We illustrate our findings both, by finding generalizations of known results, and proving new results. For the latter, we work on processes with fixed times of discontinuity.
\end{abstract}

 \tableofcontents

\section{Introduction}
This article deals with one of the  classical questions in probability theory: Limit theorems for stochastic processes. Starting with the work of  Prokhorov \cite{doi:10.1137/1101016}, limit theorems are formulated via weak convergence of probability measures on function spaces, such as the Wiener or the Skorokhod space. Prokhorov's method for proving weak convergence consists of three steps: Verifying tightness, showing the convergence of the finite dimensional distributions and explaining that these determine the limit. The first part is well-studied and the final part is essentially trivial. In contrast, verifying convergence of the finite dimensional distributions is often hard and sometimes even impossible. This difficulty motivated the necessity to develop further techniques for proving limit theorems. One of the most successful strategies is the \emph{martingale problem method} initiated by Stroock and Varadhan \cite{SV}. Instead of studying the finite dimensional distributions, the idea is to deduce the martingale property of certain test processes (which typically are functionals of the process to be described, and usually abbreviated $\Y$ in the sequel) from weak convergence and to show that these martingale properties characterize the limiting law uniquely. 

Originally, Stroock and Varadhan developed their method for It\^o diffusions, see also Example \ref{ex: SV} below. Later, Ethier and Kurtz \cite{EK} generalized it to a Markovian framework with Polish state space \(E\), i.e. they considered a martingale problem described by test processes of the type 
\begin{align}\label{eq: TP EK}
f (X) - \int_0^\cdot g (X_s) ds, \quad (f, g) \in A \subset C_b(E) \times C_b(E),
\end{align}
where $A$ is sometimes referred to as the (pre-)generator of $X$, see also Example \ref{ex: EK} below. In the spirit of the general theory of stochastic processes, Jacod and Shiryaev \cite{JS} extended the martingale problem method to semimartingales, which have not necessarily Markovian dynamics, by relating martingale properties to the so-called characteristics of the semimartingale. The path space used in the monographs \cite{EK,JS} is the Skorokhod space of \cadlag functions. 

Recently, there is  growing interest in processes which are not covered in \cite{EK,JS}.
Examples for such are solutions to stochastic partial differential equations (SPDEs) and Volterra equations (VSDEs), or certain processes with fixed times of discontinuity, which generate more and more attention in applied probability (see, e.g. \cite{10.1214/16-ECP6,10.1214/19-AAP1483}) and mathematical finance (see, e.g. \cite{BelangerShreveWong2004,Fon2020,gehmlich2018dynamic,Merton1974}). 
Providing a more concrete example, fixed times of discontinuity arise naturally in the context of processes in random environment via the quenched perspective, i.e. when the random environment is fixed, see Section~\ref{sec: motivation} below for more details.
The presence of fixed times of discontinuities turns out to be a major difficulty when one works with the classical Skorokhod space. The problem stems from the fact that continuity properties of the test processes are needed to apply the continuous mapping theorem. To see where this issue comes into play, consider the natural generalization of \eqref{eq: TP EK} when discontinuities are present:
\begin{align} \label{eq: test EK FTD}
f(X_t) - \int_0^t g (X_{s-}) q(ds), \quad (f, g) \in A \subset C_b(E) \times C_b(E),
\end{align}
where $q$ is some deterministic locally finite measure. Note that these random variables  are not necessarily continuous in the Skorokhod topology when \(q\) has point masses, as projections to fixed times are in general not continuous in the Skorokhod topology. Limit theorems for certain types of SPDEs and VSDEs were proved in \cite{jaber:hal-02279033, criens20convCMP, doi:10.1080/07362999708809484}. However,  for processes with fixed times of discontinuity we are not aware of any systematic study. 

~

The aim of this article is to develop a version of the martingale problem method which is flexible enough to cover existing convergence results and which can be used to establish new results e.g. for processes with fixed times of discontinuity.
We generalize the three main ingredients in the standard theory: (i) the state space, (ii) the set of test processes, i.e. we also allow different test processes as for instance in \eqref{eq: TP EK}, and (iii) we work with an extended type of weak convergence called \emph{weak-strong convergence} in the sequel, see~Section~\ref{sec: WS conv} below. 

For (i), we follow the idea to consider laws of stochastic processes as distributions of random variables in function spaces, tracing back to Prokhorov's  work. The classical choices for such path spaces are the Wiener space of continuous functions and the Skorokhod space of \cadlag functions. Almost all general results in the literature are formulated for these two cases. However, many processes of recent interest have less regular paths.
For example, the Volterra processes studied in \cite{jaber:hal-02279033} only have paths in (local) \(L^p\) spaces. 
To include these, and more general cases, we work with the minimal assumption that paths can be viewed as random variables in some Polish space which is not specified further.

For (ii), the idea to generalize the set of test processes stems from the observation that general processes of interest are neither semimartingales nor can be captured via test processes of the form \eqref{eq: TP EK}. For instance, we think of solutions to S(P)DEs with path dependent coefficients or truly \emph{local} martingale problems, where the structure of the test martingales is affected by stopping times.
To capture these examples we introduce an \emph{abstract martingale problem} (see Definition \ref{def:MP}) whose only input is an abstract set of test processes, denoted $\Y$ in the sequel. In Section \ref{sec: AMP} below we relate this martingale problem to several classical examples from the literature such as the martingale problems introduced in the monographs \cite{EK,JS,SV}.

Finally, for (iii), consider a sequence \((X^n)_{n \in \mathbb{N}}\) of processes with dynamics
\begin{equation} \label{eq: branching dynamics intro} \begin{split}
X^n = X_0 &+ \int_0^\cdot \int H^n (t, y) (\p^n - \q^n) (dt, dy), 
\end{split}
\end{equation}
where \(\p^n - \q^n\) is a compensated integer-valued random measure which allows fixed times of discontinuity, and \(H^n\) is a sufficiently measurable function such that the stochastic integral is well-defined. To fix ideas, suppose that we want to show convergences of \((X^n)_{n \in \mathbb{N}}\) to a process \(X\) which is also a stochastic integral w.r.t. a compensated random measure \(\p^* - \q^*\).
In the standard theory the sequence \((H^n)_{n \in \mathbb{N}}\) should then converge (in a suitable sense) to a limiting function \(H = H(X, t, y)\) which is Skorokhod continuous in the \(X\) variable for \(\q^*\)-a.a. pairs \((t, y)\). However, this assumption is typically too strong in the presence of fixed times of discontinuity.
For instance, it does not even hold in the state dependent case where \(H = H(X_{t-}, y)\).
To overcome this problem we would like to relax the continuity assumption by replacing the Skorokhod topology with a stronger topology (in which more functions are continuous). At the same time we might not want to drop the Skorokhod topology in general, as for instance tightness is well-studied for this topology.
To achieve this we introduce the notion of \emph{weak-strong convergence} to the martingale problem method. Weak-strong convergence is a certain type of convergence of probability measures on a product space where one coordinate converges in the usual weak sense and the other converges in a rather strong sense comparable to convergence in total variation. We learned about weak-strong convergence and its power from Jacod and M\'emin \cite{SPS_1981__15__529_0} and their work \cite{doi:10.1080/17442508108833169,JM81} on stability results for stochastic differential equations driven by semimartingales. The notion can be seen as a generalization of the more classical concept of \emph{stable convergence}, see Remark~\ref{rem: ws conv, in prob} below. Let us explain how we use weak-strong convergence to obtain limit theorems for processes with fixed times of discontinuity.
The notion allows us to introduce certain \emph{control variables} whose laws are supposed to converge in a rather strong sense. To use the continuous mapping theorem for weak-strong convergence (see Theorem \ref{coro: CMT}) it suffices to ask for continuity conditions when the values of the controls are fixed. Roughly speaking, this means that we can restrict our attention to a (randomized) subset of the underlying space on which continuity holds conditionally. For instance, with regard to \eqref{eq: branching dynamics intro}, if the coefficients \(H^n\) are (suitably) dominated we can restrict our attention to a subset where conditioned on the value of the control, which is related to the dominator of \(H^n, \p^n\) and \(\q^n\), the Skorokhod topology coincides with the stronger local uniform topology in which e.g. state dependent coefficients \(H = H(X_{t-}, y)\) are continuous for continuous \(H\). 
We stress that we do not require that almost all paths of the limit take values in a subspace of the Skorokhod space on which the Skorokhod and the local uniform topology coincide, but we allow the subspace to be random in some sense. 
This strategy also keeps the Skorokhod topology for the general mode of convergence and therefore one can use well-known results on tightness. 

Our main abstract results are Theorems \ref{theo: main1 pre} and \ref{theo: main approx}. In the former, we show the martingale property of test processes of the limiting martingale problem directly. In the latter, we verify the martingale property using approximating sequences of martingales. 
To illustrate applications of our theory we discuss a variety of examples. First of all, we show that our results cover, or even extend, several known limit theorems. More precisely, we reprove a classical theorem of Ethier and Kurtz \cite{EK} (see Section \ref{sec: rec EK}) and the stability result for Volterra processes from \cite{jaber:hal-02279033} (see Section \ref{sec: rec VSDE}). Furthermore, we localize conditions by Jacod and Shiryaev \cite{JS} which identify a weak limit as a semimartingale via its semimartingale characteristics (see Section \ref{sec: rec JS}). We think that this extension is of interest for future applications. 

Besides recovering results from the literature, we also present new results. 
First, we prove a version of the Ethier--Kurtz theorem for test processes of the type \eqref{eq: test EK FTD}, see Section \ref{sec: EK FTD} below. 
We also present a tightness condition which is tailored to such processes.
Second, we derive a stability result for semimartingales under a continuity assumption on the characteristics in the local uniform topology, see Section \ref{sec: main smg FTD} below. As latter is stronger than the classical Skorokhod topology, our result has a different scope than its counterpart from \cite{JS}. 
Furthermore, in Section \ref{sec: annealed smg FTD} below we specify our results to the annealed case where all processes are defined on the same probability space and the limit is allowed to have characteristics which also depend on the underlying space. Finally, in Section \ref{sec: ito FTD} below we present an application to It\^o processes with fixed times of discontinuity.
The proofs of all these results fully rely on the power of weak-strong convergence. 

The article is structured as follows: In Section \ref{sec: MP revisited main section} we introduce the abstract martingale problem, discuss the concept of weak-strong convergence and present our abstract main results. In the following Section \ref{sec: Revisited Literature main section} we discuss relations to existing literature, i.e. to a limit theorem by Ethier and Kurtz \cite{EK}, the stability result for Volterra SDEs from \cite{jaber:hal-02279033}, and we present an extension of a theorem by Jacod and Shiryaev \cite{JS}. Finally, in Section~\ref{sec: stab semi ftd} we prove our new convergence results for processes with fixed times of discontinuity. 
\\

\noindent
We end the introduction with a short list of notation, which we use in the paper:
\begin{enumerate}
	\item[-] An inequality up to a multiplicative constant is denoted by \(\lesssim\).
	\item[-] The extended real line is denoted by \(\overline{\mathbb{R}} \triangleq \mathbb{R} \cup \{\pm \infty\}\). The set of non-negative definite real-valued \(d\times d\) matrices is denoted by \(\mathbb{S}^d_+\).
	\item[-] The Lebegue measure is denoted by \(\llambda\).
	\item[-] For a Polish space \(E\) we write \(C(E)\) for the space of continuous functions \(E \to \mathbb{R}\), \(B(E)\) for the space of bounded Borel functions \(E \to \mathbb{R}\) and \(C_b(E) \triangleq C(E) \cap B(E)\).
	\item[-] For \(p \geq 1\) and a Banach space \((E, \|\cdot\|)\) we denote by \(L^p_\textup{loc} (\mathbb{R}_+, E)\) the space of equivalence classes of locally \(p\)-integrable functions from \(\mathbb{R}_+\) into \(E\), i.e. of Borel functions \(f \colon \mathbb{R}_+ \to E\) such that \(\int_0^t \|f (s)\|^p ds < \infty\) for all \(t > 0\). We endow \(L^p_\textup{loc}(\mathbb{R}_+, E)\) with the local \(L^p\)-norm~topology.
	\item[-] We write \(C^2 (\mathbb{R}^d)\) for the space of twice continuously differentiable functions \(\mathbb{R}^d \to \mathbb{R}\), \(C_c^2(\mathbb{R}^d)\) for its subspace of functions with compact support, and \(C^2_b (\mathbb{R}^d)\) for the set of bounded functions \(f \in C^2(\mathbb{R}^d)\) with bounded gradient \(\nabla f\) and bounded Hessian \(\nabla^2 f\). 
	\item[-] For an operator \(\sigma\) we write \(\on{tr} (\sigma)\) for its trace and \(\sigma^*\) for its adjoint.
	\item[-] For a Polish space \(E\) we denote the space of continuous functions \(\mathbb{R}_+ \to E\) by \(\C(E)\) and the space of \cadlag functions \(\mathbb{R}_+ \to E\) by \(\D(E)\). 
	\item[-] On a function space \(F\) the identity is denoted by \(\X \colon F \to F\). In particular, if \(F = \C(E)\) or \(\D(E)\), then \(\X\) is the so-called \emph{coordinate process}. 
	\item[-] For a \cadlag process \(Z\) we write \(\Delta Z_t \triangleq Z_t - Z_{t-}\) for its time \(t\) jump. Moreover, we denote the quadratic variation process by \([\cdot, \cdot]\).
	\item[-] For an integer-valued random measure \(\p\) with compensator \(\q\) and a suitable measurable function \(H = H( \omega, t, y)\) we write 
	\[
	H * \p_t \triangleq \int_0^t \int H (s, y) \p (ds, dy), \quad t \in \mathbb{R}_+,
	\]
	and \(H * (\p - \q)\) for the integral process of \(H\) w.r.t. the compensated random measure \(\p - \q\), cf. \cite[Section II.1.d]{JS}. Furthermore, we denote by \(G_\textup{loc} (\p)\) the set of functions which are integrable w.r.t. \(\p - \q\), see \cite[Definition II.1.27]{JS}. For a semimartingale \(Z\) we denote the set of \(Z\)-integrable processes by \(L(Z)\), cf. \cite[Section~III.6]{JS}. For all unexplained terminology related to the \emph{general theory of stochastic processes} we refer to \cite[Chapter~I]{JS}.
\end{enumerate}

\section{The Martingale Problem Method Revisited} \label{sec: MP revisited main section}

\subsection{Abstract Martingale Problems} \label{sec: AMP}

The purpose of this section is to introduce an \emph{abstract martingale problem}. Let \((\Omega, \cF, \bF = (\cF_t)_{t \geq 0})\) be a filtered space which supports a set \(\Y\) of \(\overline{\mathbb{R}}\)-valued right-continuous adapted processes.

\begin{definition}\label{def:MP}
	We call a probability measure \(P\) on \((\Omega, \mathcal{F})\) a \emph{solution to the (local) martingale problem (MP) \(\Y\)}, if all processes in \(\Y\) are (local) \((\F, P)\)-martingales.
	The sets of solutions to the martingale problem and the local martingale problem are denoted by \(\cM (\Y)\) and \(\cM_\textup{loc} (\Y)\), respectively.
\end{definition}

We now collect a variety of important examples for martingale problems. 

\begin{example}[Martingale Problem of Stroock and Varadhan] \label{ex: SV}
	Let \(\Omega = \C (\mathbb{R}^d)\) 
	and let \(\X\) be the coordinate process on \(\Omega\).
	Furthermore, let \(b \colon \mathbb{R}_+ \times \mathbb{R}^d \to \mathbb{R}^d\) and \(\sigma \colon \mathbb{R}_+ \times \mathbb{R}^d \to \mathbb{R}^{d \times r}\) be locally bounded Borel functions. 
	To obtain the  martingale problem of Stroock and Varadhan~\cite{SV} define \(\Y\) to be the set of the following processes:
	\[
	f(\X) - f(\X_0) - \int_0^\cdot \big( \langle b(s, \X_s), \nabla f(\X_s)\rangle + \tfrac{1}{2} \operatorname{tr} (\sigma \sigma^* (s, \X_s) \nabla^2 f(\X_s)) \big) ds
	\]
	where \(f \in C^2_c (\mathbb{R}^d)\).
	It is classical (see, e.g. \cite[Section 5.4]{KaraShre}) that the set \(\cM(\Y)\) coincides with the set of solution measures (i.e. laws of solution processes) for the SDE
	\[
	d X_t = b(t, X_t) dt + \sigma(t, X_t) d W_t, 
	\]
	where \(W\) is an \(r\)-dimensional standard Brownian motion.
\end{example}

\begin{example}[Martingale Problem of Ethier and Kurtz] \label{ex: EK}
	Let \(E\) be a Polish space and take \(\Omega = \D (E)\) or \(\C (E)\). %, where \(\D(E)\) denotes the set of \cadlag functions \(\mathbb{R}_+ \to E\). 
	Again, let \(\X\) be the coordinate process. Fix a set \(A \subset C_b(E) \times B(E)\). To obtain the martingale problem of Ethier and Kurtz \cite{EK} define \(\Y\) to be the set of the following processes:
	\[
	f(\X) - f(\X_0) - \int_0^\cdot g (\X_s) ds, \quad (f, g) \in A.
	\]
	The setting from Example \ref{ex: SV} is a special case of this framework.
\end{example}

\begin{example}[Semimartingale Problems] \label{ex: SMP}
	In the following we discuss two ways to characterize the laws of semimartingales via martingale problems. The first is given by \cite[Theorem III.2.7]{JS}: 
	Set \(\Omega = \D (\mathbb{R}^d)\) and let \((B, C, \nu)\) be a candidate triplet for semimartingale characteristics corresponding to a fixed truncation function \(h \colon \mathbb{R}^d \to \mathbb{R}^d\), see \cite[Definition II.2.6]{JS} for a precise definition including the technical requirements.
	Let \(\X\) be the coordinate process and define
	\begin{align*}
	\X (h) &\triangleq \X - \sum_{s \leq \cdot} \big(\Delta \X_s - h(\Delta \X_s)\big),\\
	M(h) &\triangleq \X(h) - \X_0 - B,\\
	\widetilde{C}^{ij} &\triangleq C^{ij} + \int h^i(x) h^j(x) \nu([0, \cdot] \times dx) - \sum_{s \leq \cdot} \Delta B^i_s \Delta B^j_s.
	\end{align*}
	Further, let \(\mathscr{C}^+(\mathbb{R}^d)\) be a family of bounded real-valued Borel functions on \(\mathbb{R}^d\) vanishing around the origin, which is measure determining for the class of Borel measures \(\eta\) on \(\mathbb{R}^d\) with the properties \(\eta(\{0\}) = 0\) and \(\eta(\{x \in \mathbb{R}^d \colon \|x\| > \varepsilon\}) < \infty\) for all \(\varepsilon > 0\), cf. \cite[II.2.20]{JS} for more details.
	Let \(\Y\) consist of the following processes:
	\begin{enumerate}
		\item[(i)] \(M (h)^{(i)}, i = 1, \dots, d\).
		\item[(ii)] \(M (h)^{(i)} M(h)^{(j)} - \widetilde{C}^{(ij)}, i, j = 1, \dots, d.\)
		\item[(iii)] \(\sum_{s \leq \cdot} g(\Delta \X_s) - \int g(x) \nu([0, \cdot] \times dx), g \in \mathscr{C}^+(\mathbb{R}^d)\).
	\end{enumerate}Then, 
	\(\cM_\textup{loc}(\Y)\) is the set of laws of semimartingales with characteristics \((B, C, \nu)\). % and initial law \(\eta\).
	
	Next, we discuss an alternative characterization which can be seen as a reformulation of \cite[Theorem II.2.42]{JS}.
	It is well-known (\cite[Proposition II.2.9]{JS}) that \((B, C, \nu)\) can be decomposed as follows:
	\[
	d B_t = b_t d A_t, \quad d C_t = c_t d A_t, \quad \nu(dt, dx) = F_t (dx) d A_t,
	\]
	where \(A\) is an increasing right-continuous predictable process, and \(b, c\) and \(F\) are the predictable densities of \((B, C, \nu)\) w.r.t. the induced measure \(d A_t\).
	For \(f \in C^2_b(\mathbb{R}^d)\) we set 
	\begin{align*}
	\mathcal{L} f (s) \triangleq \langle b_s, &\nabla f (\X_{s-}) \rangle + \tfrac{1}{2} \operatorname{tr} ( c_s \nabla^2 f (\X_{s-}) ) 
	\\&+ \int \big(f(\X_{s-} + x) - f(\X_{s-}) - \langle \nabla f(\X_{s-}), h(x) \rangle \big) F_s (dx).
	\end{align*}
	Let \(\Y^*\) be the set of the following processes:
	\begin{align*}
	f(\X) - f(\X_0) - \int_0^\cdot \mathcal{L} f (s) d A_s, \quad f \in C^2_b(\mathbb{R}^d).
	\end{align*}
	Then, \(\cM_\textup{loc}(\Y^*) = \cM_\textup{loc}(\Y)\).
	
	Finally, let us relate the local MP \((\Y)\) to the class of diffusions and the martingale problem of Stroock and Varadhan as explained in Example \ref{ex: SV}. For Brownian motion, more generally for diffusions, it is well-known that it suffices to consider linear and quadratic test functions, i.e. \(f(x) = x^{(i)}\) and \(f (x) = x^{(i)} x^{(j)}\) for \(i, j = 1, \dots, d\), provided one asks in addition for continuous paths, cf. \cite[Proposition 5.4.6]{KaraShre}. For Brownian motion this observation is precisely L\'evy's characterization. Namely, using \(f(x) = x^{(i)}\) yields that \(\X\) is a continuous\footnote{Here, the additional requirement of continuous paths has to be taken into consideration.} local martingale, and using in addition \(f(x) = x^{(i)}x^{(j)}\) implies that \([\X, \X] = \operatorname{Id}.\) The set \(\Y\) generalizes this idea to general semimartingales. Thereby, the processes in (iii) take care of the jump structure. To see this, assume that \(\nu = 0\), which means that (iii) consists of the processes \(\sum_{s \leq \cdot} g(\Delta \X_s)\) with \(g \in \mathscr{C}^+(\mathbb{R}^d)\). It is clear that this class consists of local martingales if and only if \(\X\) is a.s. continuous. This observation relates the processes in (iii) above to the requirement of continuous paths in L\'evy's characterization.
\end{example}

\begin{remark}
	It may happen that a probability measure solves the martingale problems from Examples \ref{ex: EK} and \ref{ex: SMP} but not both uniquely. For instance, suppose that \(X\) is a Brownian motion sticky at the origin, i.e. \(X\) solves the system
	\[
	d X_t = \1_{\{X_t \not = 0\}} d W_t,\qquad \1_{\{X_t = 0\}} dt = \tfrac{1}{\mu} d L^0_t (X), \qquad \mu > 0,
	\]
	where \(W\) is a standard Brownian motion and \(L^0 (X)\) denotes the semimartingale (right) local time of \(X\) in the origin. This characterization of a sticky Brownian motion is taken from \cite{doi:10.1080/17442508.2014.899600}. 
	It is obvious that \(X\) is a continuous local martingale (and hence a semimartingale) with quadratic variation
	\[
	[X, X] = \int_0^\cdot \1_{\{X_s \not = 0\}} ds.
	\]
	Thus, independent of the parameter \(\mu\), the law of \(X\) solves the (semi)martingale problem from Example \ref{ex: SMP} with \((0, C, 0)\) where
	\[
	C (\omega) = \int_0^\cdot \1_{\{\omega (s) \not = 0\}} ds, \quad \omega \in \D(\mathbb{R}).
	\]
	In fact, even the Wiener measure solves this martingale problem.
	We conclude that the law of \(X\) cannot be captured in a unique manner by the semimartingale problem but, as \(X\) is a one-dimensional diffusion in the sense of It\^o and McKean \cite{ItoMcKean}, its law is a unique solution to the martingale problem of Example \ref{ex: EK} when \(A \subset C_b(\mathbb{R}) \times C_b(\mathbb{R})\) is chosen appropriately,\footnote{see Section 2.7 in \cite{Freedman} for details on how \(A \subset C_b(\mathbb{R}) \times C_b(\mathbb{R})\) can be taken to capture diffusions in the sense of It\^o and McKean} see the discussion on p. 994 in \cite{doi:10.1080/17442508.2014.899600} and \cite[Remark 5.3]{10.1214/EJP.v19-2350} for more details.
\end{remark}

\begin{example}[Martingale Characterization for SPDEs] \label{ex: SPDE}
	We now describe a martingale problem for the semigroup approach to semilinear stochastic partial differential equations (SPDEs). The standard reference for this framework is the monograph of Da Prato and Zabczyk~\cite{DePrato}.
	
	Let \(E = (E, \langle \cdot, \cdot \rangle_E)\) be a separable real Hilbert space and set \(\Omega = \C (E)\).
	Take another separable real Hilbert space \((H, \langle \cdot, \cdot \rangle_H)\) and denote by \(L (H, E)\) the space of linear bounded operators \(H \to E\).
	Moreover, let \(\mu \colon \mathbb{R}_+ \times \Omega \to E\) and \(\sigma \colon \mathbb{R}_+ \times \Omega \to L (H, E)\) be progressively measurable processes. To be precise, we mean that \(\sigma h \colon \mathbb{R}_+ \times \Omega \to E\) is progressively measurable for every \(h \in H\). 
	Finally, let \(A \colon D(A) \to E\) be the generator of a \(C_0\)-semigroup on \(E\) and let \(A^* \colon D(A^*) \to E\) be its adjoint. 
	Define \(\Sigma\) to be the set of all functions \(g (\langle \cdot, y^*\rangle_E)\) where \(y^* \in D(A^*)\) and \(g \in C^2(\mathbb{R})\). For \(f = g(\langle \cdot, y^*\rangle_E) \in \Sigma\) we set
	\begin{align*}
	(\mathcal{L} f)_s \triangleq g' (\langle \X_s, y^*\rangle_E ) &(\langle \X_s, A^* y^* \rangle_E + \langle \mu_s (\X), y^*\rangle_E) \\&+ \tfrac{1}{2} g'' (\langle \X_s, y^*\rangle_E) \langle \sigma^*_s (\X) y^*,  \sigma^*_s (\X) y^*\rangle_H.
	\end{align*}
	Let \(\Y\) be the set of the following processes:
	\[
	f(\X) - f(\X_0) - \int_0^\cdot (\mathcal{L}f)_s ds, \quad f \in \Sigma.
	\]
	Then, under suitable assumptions on the coefficients \(A, b\) and \(\sigma\), the set \(\cM_\textup{loc}(\Y)\) coincides with the set of laws of mild solutions to the SPDE 
	\[
	d X_t = (A X_t + \mu_t (X) ) dt + \sigma_t (X) d W_t, %\quad X_0 \sim \eta,
	\]
	where \(W\) is a standard cylindrical Brownian motion. We refer to \cite[Proposition 2.6, Lemma~3.6]{criensritter20} for more details.
\end{example}

\begin{example}[Local Martingale Problem for SDEs of Volterra type] \label{ex: volterra}
	Let \((X, Z)\) be a measurable process with paths in \(L^p_\textup{loc} (\mathbb{R}_+, \mathbb{R}^d) \times \D(\mathbb{R}^k)\) such that, on its underlying filtered probability space, \(X\) is predictable, \(Z\) is a semimartingale with characteristics 
	\[
	B^Z = \int_0^\cdot b(X_s) ds, \quad C^Z = \int_0^\cdot c (X_s)ds, \quad \nu^{Z} (dx, dt) = \nu (X_t, dx) dt, \]
	corresponding to a fixed truncation function \(h \colon \mathbb{R}^k \to \mathbb{R}^k\), and
	\[
	X_t = g_0 (t) + \int_0^t K_{t-s} d Z_s, \quad t \in \mathbb{R}_+,
	\]
	where \(K\) is a convolution kernel \(\mathbb{R}_+ \to \mathbb{R}^{d \times k}\).
	We call such a process \((X, Z)\) a solution to a \emph{Volterra SDE (VSDE)}.
	Recently, it was proven in \cite{jaber:hal-02279033} that solutions to VSDEs have a martingale characterization.
	For \(f \in C^2_b (\mathbb{R}^k)\) and \((x, z) \in \mathbb{R}^d \times \mathbb{R}^k\) we set 
	\begin{align*}
	\mathcal{L} f (x, z) \triangleq \langle b(x)&, \nabla f (z)\rangle + \tfrac{1}{2} \operatorname{tr} ( a(x) \nabla^2 f (z)) \\&+ \int \big( f (z + y) - f(z) - \langle h (y), \nabla f (z) \rangle \big) \nu(x, dy).
	\end{align*}
	Then, \((X, Z)\) is a solution to the VSDE described above if and only if the processes 
	\[
	f (Z) - \int_0^\cdot \mathcal{L} f (X_s, Z_s) ds, \quad f \in C^2_b (\mathbb{R}^k), 
	\]
	are local martingales and 
	\[
	\int_0^t X_s ds = \int_0^t g_0(s) ds + \int_0^t K_{t-s} Z_s ds, \quad t \in \mathbb{R}_+.
	\]
	In Section \ref{sec: rec VSDE} below we take a closer look at VSDEs.
\end{example}

\subsection{Weak-Strong Convergence} \label{sec: WS conv}
In this section we recall the notion of \emph{weak-strong convergence} of probability measures on a product space, which was studied in \cite{SPS_1981__15__529_0}, see also~\cite{doi:10.1080/17442508108833169,JM81}.

Let \((U, \mathcal{U})\) be a measurable space and let \((F, \mathcal{B}(F))\) be a Polish space with its Borel \(\sigma\)-field. We define the product space \(S \triangleq U \times F\) and the corresponding product \(\sigma\)-field \(\mathcal{S} \triangleq \mathcal{U} \otimes \mathcal{B}(F)\).
Let \(C_S\) be the set of bounded \(\mathcal{S}/\mathcal{B}(\mathbb{R})\) measurable functions \(f \colon S \to \mathbb{R}\) such that \(\omega \mapsto f (\alpha, \omega)\) is continuous (as a function on \(F\)) for every \(\alpha \in U\). 
\begin{definition} \label{def: ws conv}
	Let \(P, P^1, P^2, \dots\) be probability measures on \((S, \mathcal{S})\). We say that the sequence \((P^n)_{n \in \mathbb{N}}\) \emph{converges in the weak-strong sense} to \(P\), written \(P^n \to_{ws} P\), if 
	\[
	E^{P^n} [ f ] \to E^P [ f ] \text{ as \(n \to \infty\) for all } f \in C_S.
	\]
\end{definition}
\begin{remark} \label{rem: ws conv, in prob}
	\begin{enumerate}
		\item[\textup{(i)}]
		Let \(P, P^1, P^2, \dots\) be probability measures on \(F\), let \(U\) be a singleton and extend \(P,P^1, P^2, \dots\) to the product space \(S = U \times F\) in the obvious manner. Then, it is clear that \(P^n \to_{ws} P\) if and only if \(P^n \to P\) weakly in the usual sense. This simple observation explains that weak-strong convergence is a natural extension of the usual weak convergence with an additional \emph{control variable}. In particular, any strategy to identify weak-strong limits is also a strategy to identify weak limits in the usual sense.
		\item[\textup{(ii)}]
		Weak-strong convergence has a close relation to the notion of \emph{stable convergence}, which is more commonly known in probability literature, see \cite[Section VIII.5.c]{JS}. To be more precise, if \(E\) is a Polish space and \((\Omega', \mathcal{F}', P')\) is a probability space which supports \(E\)-valued random variables \(Z^1, Z^2, \dots\), then \((Z^n)_{n \in \mathbb{N}}\) converges \emph{stably} (in the sense of \cite[Definition VIII.5.28]{JS}) if and only if the sequence
		\[
		P_n (d \omega, dz) \triangleq \delta_{Z^n(\omega)} (dz) P'(d \omega), \quad n \in \mathbb{N},
		\]
		converges in the weak-strong sense as a probability measure on \((\Omega' \times E, \mathcal{F}' \otimes \mathcal{B}(E))\), see also \cite[Proposition 2.4]{SPS_1981__15__529_0}. In certain cases stable (and therefore also weak-strong) convergence is equivalent to convergence in probability. More precisely, the sequence \((P_n)_{n \in \mathbb{N}}\) as above converges in the weak-strong sense to 
		\(
		P (d \omega, d z) = \delta_{Z (\omega)} (dz) P'(d \omega)
		\)
		if and only if \(Z^n \to Z\) in probability, see \cite[Proposition 3.5]{SPS_1981__15__529_0}. 
		\item[\textup{(iii)}]
		Yet another point of view on weak-strong convergence is the following: Let \(P_1, P_2, \dots\) be probability measures on \((S, \mathcal{S})\) with the same \(U\)-marginal \(\mu\). It is well-known that there exist transition kernel \(K_1, K_2, \dots\) such that
		\[
		P_n (d u, d f ) = K_n (u, df) \mu (du), \quad n \in \mathbb{N}.
		\]
		Then, \((P_n)_{n \in \mathbb{N}}\) converges in the weak-strong sense if and only if for every \(f \in C_b (F)\) the sequence \(K_1 f, K_2 f, \dots\) converges weakly (in the Banach space sense) in \(L^1 (U, \mathcal{U}, \mu)\).
		
	\end{enumerate}
\end{remark}

Let \(M_{mc}(S)\) be the space of all probability measures on \((S, \mathcal{S})\) endowed with the weakest topology such that the map \(P \mapsto E^P [f]\) is continuous for every \(f \in C_S\). 
Of course, \(P^n \to P\) in \(M_{mc} (S)\) if and only if \(P^n \to_{ws} P\). 
Let \(M_m (U)\) be the space of probability measures on \((U, \mathcal{U})\) endowed with the weakest topology such that the map \(P \mapsto E^P [f]\) is continuous for every bounded \(\mathcal{U}/\mathcal{B}(\mathbb{R})\) measurable function \(f \colon U \to \mathbb{R}\). 

\begin{remark} \label{rem: topology M_m(U)}
	By \cite[Proposition 2.4]{SPS_1981__15__529_0}, the above topology on \(M_m(U)\) is also the weakest topology such that the map \(P \mapsto P(A)\) is continuous for every \(A \in \mathcal{U}\).
\end{remark}

Furthermore, let \(M_c (F)\) be the space of probability measures on \((F, \mathcal{B}(F))\) endowed with the usual weak topology, i.e. the weakest topology such that the map \(P \mapsto E^P [f]\) is continuous for every bounded continuous function \(f \colon F \to \mathbb{R}\).
For \(P \in M_{mc}(S)\) we write \(P_U\) for its \(U\)-marginal and \(P_F\) for its \(F\)-marginal, respectively. To be more precise, we have the following:
\[
P_U (du) \triangleq P(du \times F), \qquad P_F (d f) \triangleq P(U \times d f).
\]
To understand the concept of weak-strong convergence better, we recall the following:

\begin{theorem}[Corollary 2.9 in \cite{SPS_1981__15__529_0}] \label{theo: ws equivalent chara}
	Suppose that \(U\) is Polish and that \(\mathcal{U}\) is its Borel \(\sigma\)-field.
	Then, \(P^n \to_{ws} P\) if and only if \(\{P^n_U\colon n \in \mathbb{N}\}\) is relatively compact in \(M_m (U)\) and \(P^n \to P\) in~\(M_c (S)\).
\end{theorem}

In other words, weak-strong and weak convergence on \(S\) distinguish by relative compactness of the first marginales in \(M_m(U)\).

The following result shows that relative compactness in \(M_{mc}(S)\) is equivalent to relative compactness of the marginales.
\begin{theorem} [Theorem 2.8 in \cite{SPS_1981__15__529_0}] \label{theo: ws conv rela comp} 
	A set \(I \subset M_{mc}(S)\) is relatively compact if and only if the sets \(\{P_U \colon P \in I\}\) and \(\{P_F \colon P \in I\}\) are relatively compact in \(M_m(U)\) and \(M_{c}(F)\), respectively.
\end{theorem}

In non-metrizible spaces, compactness does not imply sequential compactness. Thus, as one would like to work with sequences, it is interesting to have a condition for the metrizibility of \(M_{mc}(S)\).

\begin{theorem}[Proposition 2.10 in \cite{SPS_1981__15__529_0}] \label{theo: Mmc metrizible}
	If \(\mathcal{U}\) is separable, then \(M_{mc}(S)\) and \(M_m (U)\) are metrizible.
\end{theorem}

Thus, in case \(\mathcal{U}\) is separable, any sequence in a relatively compact subset of \(M_{mc}(S)\) has a convergent subsequence. The following theorem gives a criterion for the existence of a convergent subsequence even when \(\mathcal{U}\) is not separable.
\begin{theorem}[Theorem 2.8 in \cite{JM81}] \label{theo: nice version rel comp}
	For each \(n \in \mathbb{N}\) let \(P^n \in M_{mc}(S)\) and \(Q\in M_c(U).\)
	Suppose that \(\{P^n_F \colon n \in \mathbb{N}\}\) is relatively compact in \(M_c(F)\) and that \(P^n_U \equiv Q\) for all \(n \in \mathbb{N}\). Then, there exist a probability measure \(P \in M_{mc}(S)\) and a subsequence \((P^{n_m})_{m \in \mathbb{N}}\) such that \(P^{n_m} \to_{ws} P\) and~\(P_U = Q\).
\end{theorem}

Next, we recall a continuous mapping theorem for weak-strong convergence.
For \(A \in \mathcal{S}\) and \(\alpha \in U\) we write
\[
A_\alpha \triangleq \big\{ \omega \in F \colon (\alpha, \omega) \in A\big\} \in \mathcal{B}(F).
\]
\begin{definition} \label{def: continuous ws}
	An \(\mathcal{S}/\mathcal{B}(\mathbb{R})\) measurable function \(g \colon S \to \mathbb{R}\) is called \emph{\((P^n, P)\)-continuous} if there exists a set \(A \in \mathcal{S}\) such that 
	\begin{enumerate}
		\item[\textup{(i)}] \(P^n(A) \to 1\) as \(n \to \infty\), and \(P(A) = 1\).
		\item[\textup{(ii)}] The set   \[
		\{ (\alpha, \omega) \in A \colon A_\alpha \ni \zeta \mapsto g (\alpha, \zeta) \text{ is discontinuous at } \omega\}
		\]
		is \(P\)-null.
	\end{enumerate}
\end{definition}
The following partial version of the Portmanteau theorem can be used to check part (i) in Definition \ref{def: continuous ws}.
\begin{proposition}[Proposition 2.11 in \cite{SPS_1981__15__529_0}] \label{prop: limsup ws conv}
	If \(P^n \to_{ws} P\), then \(\limsup_{n \to \infty} P^n (G) \leq P(G)\) for all \(G \in \mathcal{S}\) such that \(G_\alpha\) is closed in \(F\) for every \(\alpha \in U\).
\end{proposition}

\begin{theorem}[Theorem 2.16 in \cite{SPS_1981__15__529_0}] \label{coro: CMT}
	Suppose that \(P^n \to_{ws} P\) and let \(g \colon S \to \mathbb{R}\) be \((P^n, P)\)-continuous such that
	\[\sup_{n \in \mathbb{N}} E^{P^n} [ |g| \1_{\{|g| > a\}} ] \to 0\] as \(a \to \infty\).
	Then, \(E^{P^n} [g] \to E^P [g]\) as \(n \to \infty\).
\end{theorem}

At the end of this section we introduce a useful component to build a set \(A\) as in the definition of \((P^n, P)\)-continuity.
In the following, let \((E, r)\) be a Polish space\footnote{\(r\) is the corresponding metric.} and let \(k \colon \mathbb{R}_+ \to \mathbb{R}_+\) be a Borel function such that for every \(t > 0\)
\begin{align}\label{eq: k prop}
\lim_{\varepsilon \searrow 0} \sup \big\{ k(s) \colon s \not = t, t - \varepsilon \leq s \leq t + \varepsilon\big\} = 0.
\end{align}
\begin{lemma} \label{lem: set G finite jump}
	The property \eqref{eq: k prop} holds if and only if for every \(T, a > 0\) there exists no \(t > 0\) such that the set \(\{s \in [0, T] \colon k(s) \geq a\}\) contains a sequence \((t_n)_{n \in \mathbb{N}}\) with \(t_n \not = t\) and \(t_n \to t\) as \(n \to \infty\). In particular, \eqref{eq: k prop} holds if \(\{s \in [0, T] \colon k(s) \geq a\}\) is finite for all~\(T, a > 0\). 
\end{lemma}
\begin{proof}
	Let us start with the \emph{if} implication. Take \(t > 0\) and assume for contradiction that there exists a sequence \(\varepsilon_n \searrow 0\) and a constant \(a > 0\) such that \(
	\sup \{ k(s) \colon s \not = t, t - \varepsilon_n \leq s \leq t + \varepsilon_n\} > a 
	\)
	for all \(n \in \mathbb{N}\). There exists an \(s_1 \not = t, t - \varepsilon_1 \leq s_1 \leq t + \varepsilon_1\) such that \(k(s_1) \geq a\). Then, choose \(N \in \mathbb{N}\) such that \(s_1 \not \in [t - \varepsilon_N, t + \varepsilon_N]\) and take \(s_2 \not = t, t - \varepsilon_N \leq s_2 \leq t + \varepsilon_N\) such that \(k(s_2) \geq a\). Proceeding in this manner we get a sequence \(s_1, s_2, \dots\) with \(s_n \not = t, s_n \to t\) and \(k (s_n) \geq a\).
	This is a contradiction and the if implication follows.
	
	We now prove the \emph{only if} implication. For contradiction, assume that \(T, a, t > 0\) are such that there exists a sequence \(t_1, t_2, \dots \in [0, T]\) such that \(t_n \not = t, t_n \to t\) and \(k(t_n) \geq a\). Note that \(|t - t_n| \not = 0\) and that 
	\[
	t - |t - t_n| = |t - t_n + t_n| - |t - t_n| \leq t_n = |t + t_n - t| \leq t + |t - t_n|.
	\]
	Hence, 
	\[
	\liminf_{n \to \infty}  \sup \big\{ k(s) \colon s \not = t, t - |t - t_n| \leq s \leq t + |t - t_n|\big\} \geq \liminf_{n \to \infty} k(t_n) \geq a.
	\]
	As this is a contradiction, the only if implication is also proved.
\end{proof}

\begin{remark}
	It is possible that \eqref{eq: k prop} holds while \(\{t \in [0, T] \colon k(t) \geq a\}\) is infinite for some \(T, a  > 0\). Indeed, take for instance \(k (t) = \sum_{k = 1}^\infty \1_{\{t = 1/k\}}\). 
\end{remark}

To motivate what comes next, suppose that \(\omega_1, \omega_2, \dots \in \D(E)\) is a sequence whose jumps are controlled by \(k\), i.e. \(r (\omega_n (t), \omega_n(t-)) \leq k(t)\) for all \(t > 0\) and \(n \in \mathbb{N}\). Furthermore, suppose that \(\omega \in \D(E)\) is such that \(\omega_n \to \omega\) in the Skorokhod \(J_1\) topology. By standard properties of this topology, for every \(t > 0\) with \(r(\omega(t), \omega(t-)) > 0\) there exists a sequence \(t_n \to t\) such that \(r (\omega_n (t_n), \omega_n (t_n - )) \to r (\omega (t), \omega (t-))\). W.l.o.g. we can assume that there exists an \(a > 0\) such that \(r (\omega_n (t_n), \omega_n(t_n - )) \geq a\) for all \(n \in \mathbb{N}\). By hypothesis, \(k (t_n) \geq r (\omega_n (t_n), \omega_n (t_n-)) \geq a\).  Recalling Lemma \ref{lem: set G finite jump}, we must have \(t_n = t\) for all large enough \(n\). In summary, we obtain convergence of the jumps, i.e. \(r(\omega_n (t), \omega_n (t-)) \to r(\omega(t), \omega (t-))\). This property is certainly necessary for local uniform convergence and, as we will see below, it is even sufficient. Summarizing, for sequences whose jumps are controlled via \(k\) we obtain equivalence of the Skorokhod \(J_1\) and the local uniform topology. As we are mainly interested in continuity properties, this is quite useful. In the following we fill in the remaining details.

Let \(\k \colon \mathbb{R}_+ \to\mathbb{R}_+\) be increasing and continuous. Here, we do \emph{not} assume that \(\k (0) = 0\). In fact, typical examples for \(\k\) could be \(\k \equiv 1\) or \(\k(x) = 1 + x\). Finally, we fix a reference point \(x_0 \in E\).
Define
\[
G \triangleq \Big\{ \omega \in \D(E) \colon r(\omega(t), \omega(t-)) \leq k(t) \k\Big(\sup_{s \leq t}r(\omega(s), x_0)\Big)\text{ for all } t > 0\Big\}. 
\]
Using an exhausting sequence for the jumps of the coordinate process on \(\D(E)\), we immediately see that \(G\in \mathcal{B}(\D(E))\), where \(\D(E)\) is endowed with the Skorokhod \(J_1\) topology. In fact, we can say more, as the following proposition shows. 
\begin{proposition} \label{prop: A local uniform Skorokhod}
	The set \(G\) is closed in \(\D(E)\) for the local uniform and the Skorokhod \(J_1\) topology. Moreover, on \(G\) the Skorokhod \(J_1\) topology coincides with the local uniform topology.
\end{proposition}
\begin{proof}
	First of all, as \(\k\) is continuous, it is easy to see that \(G\) is closed in the local uniform topology. 
	Hence, it suffices to prove the second claim, i.e. that the Skorokhod \(J_1\) and the local uniform topology coincide on \(G\). Of course, we only need to show that Skorokhod \(J_1\) convergence implies local uniform convergence. 
	Take \(\omega_1, \omega_2, \dots \in G\) and \(\omega \in\D(E)\) such that \(\omega_n \to \omega\) in the Skorokhod \(J_1\) topology. By virtue of \cite[Theorem~2.6.2]{doi:10.1137/1101022} and \cite[Propositions VI.2.1, VI.2.7]{JS}, it suffices to prove that \(r(\omega_n (t), \omega_n(t-)) \to r(\omega(t), \omega(t-))\) for all \(t > 0\) such that \(r(\omega(t), \omega(t-)) > 0\). We fix such a \(t > 0\). Thanks to \cite[Problem~16, p. 152]{EK} (or \cite[Theorem 2.7.1]{doi:10.1137/1101022}), there exists a compact set \(K = K_t \subset E\) such that \(\omega_n (s) \in K\) for all \(s \leq t + 1\) and \(n \in \mathbb{N}\). Hence, taking into account that \(\k\) is increasing, there exists a constant \(C > 0\) such that 
	\[
	\sup_{n \in \mathbb{N}} \kappa \Big( \sup_{s \leq t + 1} r (\omega_n (s), x_0)\Big) \leq \k \Big(\sup_{x \in K} r(x, x_0)\Big) \leq C.
	\]
	It is well-known (\cite[Proposition VI.2.1]{JS}) that there exists a sequence \(t_n \to t\) with
	\(
	r(\omega_n(t_n),\) \(\omega_n(t_n -)) \to r(\omega(t), \omega(t-)).
	\)
	Now, for large enough \(n\) we get
	\(
	r(\omega_n (t_n), \omega_n(t_n-))\) \(\leq C k (t_n)
	\)
	and Lemma \ref{lem: set G finite jump} yields that \(t_n = t\) for large enough \(n\). This implies \(r(\omega_n(t), \omega_n(t-))\) \(\to r(\omega(t), \omega(t-))\) and the proof is complete.
\end{proof}
Versions of Proposition \ref{prop: A local uniform Skorokhod} for \(E = \mathbb{R}^d\) and \(\kappa \equiv 1\) or \(\kappa (x) = 1 + x\) are given as \cite[Lemma~4.2]{doi:10.1080/17442508108833169} and \cite[Lemma 3.6]{JM81}. 

In Section \ref{sec: stab semi ftd} below we use a randomized version of the set \(G\) and the continuous mapping theorem for weak-strong convergence to relax the continuity assumptions in certain stability results for semimartingales, and to derive a version of the Ethier--Kurtz stability theorem associated to test processes of the type~\eqref{eq: test EK FTD}. The randomization is important as it allows for a much more flexible jump structure than the set \(G\) might suggest.

\subsection{Identifying Weak Limits via Abstract Martingale Problems}
The classical martingale problems from Examples \ref{ex: SV}, \ref{ex: EK} and \ref{ex: SMP} proved themselves as valuable tools to identify weak limits of stochastic processes.
In the following we discuss such an application for the abstract martingale problem as introduced in Definition~\ref{def:MP}.

\subsubsection{The Setting}
We recall our setting: 

\begin{listing}
	\label{ass 1}
	Let \((\Omega, \mathcal{F}, \F)\) be a filtered space which supports a family \(\Y\) of \(\overline{\mathbb{R}}\)-valued right-continuous adapted processes. Furthermore, let \((U, \mathcal{U})\) be a measurable space and let \((E, \mathcal{B} (E))\) and \((F, \mathcal{B}(F))\) be Polish spaces with their Borel \(\sigma\)-fields. Moreover, let \(X = (X_t)_{t \geq 0}\) be a measurable \(E\)-valued process  such that for each \(\omega \in \Omega\) the process \(X(\omega)\) is an element of the Polish space \(F\) and the map \(\Omega \ni \omega \mapsto X (\omega) \in F\) is \(\mathcal{F}/\mathcal{B}(F)\) measurable. Finally, let \(L\) be a \(U\)-valued random variable on \((\Omega, \mathcal{F})\). As in Section \ref{sec: WS conv}, we also define the product space \(S \triangleq U \times F\) and \(\mathcal{S} \triangleq \mathcal{U} \otimes \mathcal{B}(F)\).
\end{listing}
\begin{example}
	\begin{enumerate}
		\item[(i)] In many classical cases
		\(X\) has \cadlag or even continuous paths and it is natural to take \(F = \D(E)\) or \(\C(E)\) (endowed with the Skorokhod \(J_1\) topology\footnote{On \(\C(E)\) the Skorokhod \(J_1\) coincides with the local uniform topology, see \cite[Problem 25, p. 153]{EK}.} which renders \(\D(E)\) and \(\C(E)\) into Polish spaces). 
		\item[(ii)] Take a separable real Banach space \(E = (E, \|\cdot\|)\) and suppose that \(X\) solves (in some sense) an SPDE with state space \(E\). For certain SPDEs it is not known whether the paths \(\mathbb{R}_+ \ni t \mapsto X_t \in E\) are \cadlag or continuous. In these cases one cannot use \(F = \D(E)\) or \(\C(E)\). The paths of \(X\) are often known to be locally \(p\)-integrable for some \(p \geq 1\), i.e. \[\int_0^t \|X_s\|^p ds < \infty, \quad t \in \mathbb{R}_+,\] and one can take \(F = L_\textup{loc}^p (\mathbb{R}_+, E)\), cf. \cite[Theorem 7.5]{DePrato}. We refer to \cite{FG95} for another example of a non-standard path space used in SPDE settings. In the same spirit, the path space \(F = L^p_\textup{loc}(\mathbb{R}_+, \mathbb{R}^d) \times \D(\mathbb{R}^k)\) is a natural state space for solutions to VSDEs as introduced in Example \ref{ex: volterra}.
	\end{enumerate}
\end{example}
\begin{definition} \label{def: determining} 
	We call a familiy \(\Z = \{\Z_t, t \in \mathbb{R}_+\}\) a \emph{determining set} (for \(\Y\)), if it has the following properties:
	\begin{enumerate}
		\item[\textup{(i)}] 
		For every \(t \in \mathbb{R}_+\), \(\Z_t\) consists of bounded \(\mathcal{S} / \mathcal{B}(\mathbb{R})\) measurable functions~\(S \to \mathbb{R}\).
		\item[\textup{(ii)}] For every probability measure $P$ on $(\Omega, \mathcal F)$, \(Y \in \Y\) and \(s < t\) with \(Y_t, Y_s \in L^1(P)\) the following implication holds:
		\begin{align*}
		E^P \big[ Y_t Z^\circ_s (L, X) \big] = E^P \big[ Y_s Z^\circ_s (L, X) \big] &\text{ for all } Z^\circ_s \in \Z_s\ 
		\Longrightarrow \ P\text{-a.s. } E^P \big[ Y_t | \cF_s \big] = Y_s.
		\end{align*}
	\end{enumerate}
\end{definition}
We now describe typical 
determining sets for two important settings.
\begin{example} \label{ex: determining set}
	\begin{enumerate}
		\item[\textup{(i)}] Suppose that
		\(\mathcal{F}_t = \sigma (X_s, s \leq t)\) for \(t \in \mathbb{R}_+\).
		Take \(F = \D(E)\) or \(\C(E)\) and denote the corresponding coordinate process by \(\X\). Then, for any dense set \(D \subset \mathbb{R}_+\) the family \(\Z = \{\Z_t, t \in \mathbb{R}_+\}\) defined by 
		\begin{align*}
		\qquad \Z_t \triangleq \Big\{ \prod_{i = 1}^n \ h_i (\X_{t_i}) \colon n \in \mathbb{N}, t_1, \dots, t_n \in D \cap [0, t], h_1, \dots, h_n \in C_b (E)\Big\}
		\end{align*}
		is determining. Let us shortly explain that (i) and (ii) in Definition \ref{def: determining} are satisfied: As \(\mathcal{B}(F) = \sigma (\X_t, t \in \mathbb{R}_+)\) (see \cite[Proposition 3.7.1]{EK}), part (i) is obvious. Part (ii) follows from the monotone class theorem.
		\item[\textup{(ii)}]
		Take \(E = \mathbb{R}, F = L^p_\textup{loc} (\mathbb{R}_+, \mathbb{R})\) for \(p \geq 1\) and let \(\X \colon F \to F\) be the identity. Furthermore, assume that \(X\) is progressively measurable w.r.t. \(\cF_t \triangleq \sigma (X_s, s \leq t)\) and
		\begin{align} \label{eq: X = liminf}
		X_t = \hat{X}_t \triangleq \liminf_{n \to \infty} \Big(n\int_{t}^{t + \frac{1}{n}} X_s ds \Big), \quad t \in \mathbb{R}_+.
		\end{align}
		By Lebesgue's differentiation theorem we always have \(X_t = \hat{X}_t\) for a.a. \(t \in \mathbb{R}_+\) and hence, as \(X\) should be considered as an \(L^p_\textup{loc}(\mathbb{R}_+, \mathbb{R})\)-valued random variable, the last assumption is essentially without loss of generality.
		Then, for any dense set \(D \subset \mathbb{R}_+\) the family \(\Z = \{\Z_t, t \in \mathbb{R}_+\}\) defined by
		\begin{align*}
		\qquad \quad \Z_t \triangleq \Big\{ \prod_{i = 1}^n \ h_i \Big(\int_0^{t_i} \X_s ds\Big) \colon n \in \mathbb{N}, t_1, \dots, t_n \in D \cap [0, t], h_1, \dots, h_n \in C_b (\mathbb{R})\Big\}
		\end{align*}
		is determining. Part (i) in Definition \ref{def: determining} follows from the fact that the maps
		\[
		L^p_\textup{loc} (\mathbb{R}_+, \mathbb{R}) \ni f = (f(s))_{s \geq 0} \mapsto \int_0^{t} f(s) ds, \quad t \in \mathbb{R}_+,
		\]
		are continuous. For (ii) it suffices to use the monotone class theorem together with the observation that
		\[
		\cF_t = \mathcal{G}_t \triangleq \sigma \Big(\int_0^s X_u du, s \leq t\Big).
		\]
		Here, the inclusion \(\mathcal{G}_t \subset \cF_t\) is clear and the converse inclusion follows from \(X = \hat{X}\). To see this, note that \(\hat{X}\) is \((\mathcal{G}_{t+})_{t \geq 0} \triangleq (\mathcal{H}_t)_{t \geq 0}\)-predictable (as pointwise limit of continuous processes) and thus \((\mathcal{H}_{t-})_{t \geq 0}\) adapted. As \(\mathcal{H}_{t-} \subset \mathcal{G}_t\), \(\hat{X}\) is \((\mathcal{G}_t)_{t \geq 0}\)-adapted and \(X = \hat{X}\) implies \(\mathcal{F}_t \subset \mathcal{G}_t\).
	\end{enumerate}
\end{example}

\begin{definition}
	We call \(\Y\) \emph{canonical}, if for every \(t \in \mathbb{R}_+\) there exists a set \(\Y^\circ_t\) of \(\mathcal{S}/\mathcal{B}(\overline{\mathbb{R}})\) measurable functions \(Y^\circ_t \colon S \to \overline{\mathbb{R}}\) such that for every \(Y \in \Y\) there exists a \(Y^\circ_t \in \Y^\circ_t\) such that \(Y_t = Y^\circ_t (L, X)\). We call \((Y^\circ_t)_{t \geq 0}\) a \emph{canonical version} of \(Y\).
\end{definition}

\begin{example} The sets \(\Y\) in Examples \ref{ex: SV}, \ref{ex: EK}, \ref{ex: SMP}, \ref{ex: SPDE} and \ref{ex: volterra} are canonical.
\end{example}

\subsubsection{The Results} \label{sec: results}
In addition to Assumption \ref{ass 1}, we assume the following:

\begin{listing}
	\label{ass 2}
	For every \(n \in \mathbb{N}\), let \(\B^n \triangleq (\Omega^n, \mathcal{F}^n, (\mathcal{F}^n_t)_{t \geq 0}, P^n)\) be a filtered probability space, which supports a \(U\)-valued random variable \(L^n\) and an \(E\)-valued measurable processes \(X^n = (X^n_t)_{t \geq 0}\) such that for every \(\omega \in \Omega\) the process \(X^n (\omega)\) is an element of \(F\) and the map \(\Omega^n \ni \omega \mapsto X^n (\omega) \in F\) is \(\mathcal{F}^n/\mathcal{B}(F)\) measurable. Moreover,  we fix a probability measure \(P\) on \((\Omega, \mathcal{F})\) and denote 
	\(
	Q^n \triangleq P^n \circ (L^n, X^n)^{-1}\) and \(Q \triangleq P \circ (L, X)^{-1}.
	\)
\end{listing}

Recall Definition \ref{def: ws conv} for the concept of weak-strong convergence, and Definition \ref{def: continuous ws} for the concept of \((Q^n, Q)\)-continuity.
\begin{theorem} \label{theo: main1 pre}
	Let Assumptions \ref{ass 1} and \ref{ass 2} hold, \(D \subset \mathbb{R}_+\) be dense, and assume the following:
	\begin{enumerate}
		\item[\textup{(A1)}] \(Q^n \to_{ws} Q\).
		\item[\textup{(A2)}] There exists a determining set \(\Z = \{\Z_t, t \in \mathbb{R}_+\}\) for \(\Y\). 
		\item[\textup{(A3)}] \(\Y\) is canonical and for every \(Y \in \Y\) there exists a canonical version \((Y^\circ_t)_{t \geq 0}\) such that for every \(t \in D, s \in D \cap [0, t]\) and \(Z^\circ_s \in \Z_s\) the following hold: \(Y^\circ_t\) and 
		\(Y^\circ_t Z^\circ_s\) are \((Q^n, Q)\)-continuous, the set 
		\(
		\{Y^\circ_r (L^n, X^n) \colon r \in D \cap [0, t], n \in \mathbb{N}\}
		\)
		is uniformly integrable, and 
		\begin{align} \label{eq: conv cond}
		\lim_{n \to \infty} E^{P^n} \big[ (Y^\circ_t (L^n, X^n) - Y^\circ_s (L^n, X^n)) Z^\circ_s (L^n, X^n) \big] = 0.
		\end{align}
	\end{enumerate}
	Then, \(P\) solves the MP \((\Y)\), i.e. \(P \in \cM(\Y)\).
\end{theorem}

Before we prove this theorem, we briefly recall \cite[Lemma IX.1.11]{JS}, which is very useful in the following.
\begin{lemma} \label{lem: UI JS}
	A family \(\{Z_i \colon i \in I\}\) is uniformly integrable if and only if \[\sup_{i \in I} E \big[ |Z_i| - |Z_i| \wedge z \big] \to 0 \text{ as } z \to \infty.\]
\end{lemma}
\begin{proof}[Proof of Theorem \ref{theo: main1 pre}]
	We have to prove that every \(Y \in \Y\) is a martingale on \((\Omega, \mathcal{F}, \bF, P)\). Fix \(Y \in \Y\) with canonical version \(Y^\circ = (Y^\circ_t)_{t \geq 0}\) and take \(s, t \in D\) such that \(s < t\) and \(Z^\circ_s \in \Z_s\). Using the canonical property of $Y$ in the first, the $(Q^n, Q)$-continuity of $Y^\circ Z^\circ$ and Theorem \ref{coro: CMT} in the second, and \eqref{eq: conv cond} in the third equality, we find
	\begin{equation}\label{eq T361}
	\begin{split}
	E^{P} \big[(Y_t - Y_s) Z^\circ_s(L, X) \big] & = E^{P} \big[(Y_t^\circ(L, X) - Y_s^\circ(L, X)) Z^\circ_s(L, X) \big]  \\ & = \lim_{n \to \infty} E^{P^n} \big[ (Y^\circ_t (L^n, X^n) - Y^\circ_s (L^n, X^n)) Z^\circ_s (L^n, X^n)\big] = 0. 
	\end{split}
	\end{equation}
	Thus, by (A2) and part (ii) of Definition \ref{def: determining}, we conclude that \(P\)-a.s.
	\(
	E^P \big[ Y_t | \cF_s \big] = Y_s.
	\)

	Next, we show this identity for general \(s < t\). We start by showing that the set \(\{Y^\circ_s (L, X) \colon s \in D \cap [0, t]\}\) is uniformly integrable for every \(t \in \mathbb{R}_+\). Let \(s \in D\).
	The \((Q^n, Q)\)-continuity of \(Y^\circ_s\) and Theorem \ref{coro: CMT} yield that
	\begin{align*}
	E^P \big[ |Y^\circ_s (L, X)| - |Y^\circ_s (L, X)| \wedge N \big] &= \lim_{n \to \infty} E^{P^n} \big[ |Y^\circ_s (L^n, X^n)| - |Y^\circ_s (L^n, X^n)| \wedge N \big]
	\\&\leq \sup_{n \in \mathbb{N}} E^{P^n} \big[ |Y^\circ_s (L^n, X^n)| - |Y^\circ_s (L^n, X^n)| \wedge N \big].
	\end{align*}
	Thus,
	\begin{align*}
	\sup_{s\in D\cap [0,t]} &E^P \big[|Y_s^\circ(L, X)| - |Y_s^\circ(L, X)| \wedge N \big] \\&\leq  \sup_{s\in D\cap [0,t]}\sup_{n \in \mathbb{N}} E^{P_n}\big[|Y_s^\circ(L^n,X^n)| - |Y_s^\circ(L^n,X^n)| \wedge N \big] \to 0 
	\end{align*}
	as $N\to\infty$ by Lemma \ref{lem: UI JS}. Another application of Lemma \ref{lem: UI JS} implies that the set \(\{Y_s \colon s \in D \cap [0, t]\} = \{Y^\circ_s (L, X) \colon s \in D \cap [0, t]\}\) is uniformly integrable.
	
	Now, let \(s < t\) be arbitrary, i.e. not necessarily in the set \(D\).
	As \(D\) is dense in \(\mathbb{R}_+\), there are sequences \(t_n \searrow t\) and \(s_n \searrow s\) in $D$ such that \(s_n < t_n\) for all \(n \in \mathbb{N}\). 
	The right-continuity of \(Y\) and Vitali's theorem yield that for every \(G \in \cF_s\) we have
	\begin{align}\label{eq T362}
	E^P \big[ Y_t \1_G \big] = \lim_{n \to \infty} E^P \big[ Y_{t_n} \1_G \big] = \lim_{n \to \infty} E^P \big[ Y_{s_n} \1_G \big] = E^P \big[ Y_s \1_G \big].
	\end{align}
	We conclude the \(P\)-martingale property of \(Y\). The proof is complete.
\end{proof}

\begin{remark}
	In case (A1) holds and \(P \in \cM (\Y)\), \eqref{eq: conv cond} has to hold under the continuity assumptions in (A3).
\end{remark}

Let us also comment on the case without control variables, which can be captured with the assumption that \(U\) is a singleton. We will simplify our notation for this situation and remove \(L^n\) and \(L\).
To clarify our terminology, we write \(X^n \to X\) weakly when the laws of \(X^n\) converge in \(M_c(F)\) to the law of \(X\). Moreover, we call a Borel function \(f \colon F \to \mathbb{R}\) to be \emph{\(P\)-continuous at \(X\)}, if there exists a set \(C \in \mathcal{B}(F)\) such that \(P(X \in C) = 1\) and \(f (s_n) \to f(s)\) whenever \(s_n \to s \in C\). The following is an immediate consequence of Theorem \ref{theo: main1 pre}.

\begin{corollary} \label{coro: main coro singleton}
	Let Assumptions \ref{ass 1} and \ref{ass 2} hold, \(U\) be a singleton, \(D \subset \mathbb{R}_+\) be dense, and assume the following:
	\begin{enumerate}
		\item[\textup{(S1)}] \(X^n \to X\) weakly.
		\item[\textup{(S2)}] There exists a determining set \(\Z = \{\Z_t, t \in \mathbb{R}_+\}\) for \(\Y\). % and for all \(t \in D\) every \(Z^\circ \in \Z_t\) is \(P\)-a.s. continuous at \(X\).
		\item[\textup{(S3)}] \(\Y\) is canonical and for every \(Y \in \Y\) there exists a canonical version \((Y^\circ_t)_{t \geq 0}\) such that for all \(t \in D, s \in D \cap [0, t]\) and \(Z^\circ_s \in \Z_s\) the following hold: \(Y^\circ_t\) and 
		\(Y^\circ_t Z^\circ_s\) are \(P\)-continuous at \(X\), the set 
		\(
		\{Y^\circ_r (X^n) \colon r \in D \cap [0, t], n \in \mathbb{N}\}
		\)
		is uniformly integrable, and 
		\begin{align*}
		\lim_{n \to \infty} E^{P^n} \big[ (Y^\circ_t (X^n) - Y^\circ_s (X^n)) Z^\circ_s (X^n) \big] = 0.
		\end{align*}
	\end{enumerate}
	Then, \(P\) solves the MP \((\Y)\), i.e. \(P \in \cM(\Y)\).
\end{corollary}

As the following proposition shows, \eqref{eq: conv cond} in Theorem \ref{theo: main1 pre} holds in case \(Y^\circ (L^n, X^n)\) can be approximated by a sequence of martingales on \(\B^n\).
\begin{proposition} \label{prop: mg approx}
	Let all assumptions from Theorem \ref{theo: main1 pre} hold, except \eqref{eq: conv cond}. Suppose that for every \(s \in D, Z^\circ_s \in \Z_s\) the random variable \(Z^\circ_s (L^n, X^n)\) is \(\mathcal{F}^n_s\)-measurable, and that there exists a sequence \((Y^n)_{n \in \mathbb{N}}\) such that \(Y^n\) is a martingale on \(\B^n\). If 
	\begin{align}
	\label{eq: conv cond 2}
	\lim_{n \to \infty} E^{P^n} \big[ (Y^n_t - Y^\circ_t (L^n, X^n)) Z^\circ_s (L^n, X^n)\big] = 0,\quad s, t \in D, s \leq t, Z^\circ_s \in \Z_s,
	\end{align}
	then \eqref{eq: conv cond} holds.
	In particular, \eqref{eq: conv cond} holds in case
	\begin{align}\label{eq: L1 cond}
	\lim_{n \to \infty}E^{P^n} \big[ |Y^n_t - Y^\circ_t (L^n, X^n)| \big] = 0.
	\end{align}
\end{proposition}
\begin{proof}
	Using the martingale property of $Y^n$ in the first, and \eqref{eq: conv cond 2} in the second equality, the hypothesis yields that 
	\begin{align*}
	\lim_{n \to \infty} &E^{P^n} \big[ (Y^\circ_t (L^n, X^n) - Y^\circ_s (L^n, X^n)) Z^\circ_s(L^n, X^n) \big] \\&= \lim_{n \to \infty} E^{P^n} \big[ (Y^\circ_t (L^n, X^n) - Y^n_t + Y^n_s - Y^\circ_s (L^n, X^n)) Z^\circ_s (L^n, X^n) \big] = 0. 
	\end{align*}
	The second claim follows from the first.
\end{proof}

\begin{remark}
	By Vitali's theorem, \eqref{eq: L1 cond} can be replaced by uniform integrability and convergence in probability: If the family \(\{|Y^n_t - Y^\circ_t (L^n, X^n)|\colon n \in \mathbb{N}\}\) is uniformly integrable and for all \(\varepsilon > 0\)
	\[
	P^n (|Y^n_t - Y^\circ_t (L^n,X^n)| \geq \varepsilon) \to 0 \text{ as }n \to \infty,
	\]
	then \eqref{eq: L1 cond} holds. 
\end{remark}
We now replace the uniform integrability assumption in (A3) by a uniform integrability assumption on the approximating martingales.
\begin{theorem} \label{theo: main approx}
	Let all assumptions from Theorem \ref{theo: main1 pre} hold, except \textup{(A3)}. Suppose that for every \(s \in D, Z^\circ_s \in \Z_s\) the random variable \(Z^\circ_s (L^n, X^n)\) is \(\mathcal{F}^n_s\)-measurable, and that the following holds:
	\begin{enumerate}
		\item[\textup{(A4)}] \(\Y\) is canonical and for every \(Y \in \Y\) there exists a canonical version \((Y^\circ_t)_{t \geq 0}\) such that for every \(t \in D, s \in D \cap [0, t]\) and \(Z^\circ_s \in \Z_s\) the following hold: \(Y^\circ_t\) and 
		\(Y^\circ_t Z^\circ_s\) are \((Q^n, Q)\)-continuous. Moreover, there exists a sequence \((Y^n)_{n \in \mathbb{N}}\) such that \(Y^n\) is a martingale on \(\B^n\), the set \(\{Y^n_s \colon s \in D \cap [0, t], n \in \mathbb{N}\}\) is uniformly integrable and 
		\begin{align} \label{eq: main conv cond mart approx}
		\lim_{n \to \infty} P^n (|Y^n_t - Y^\circ_t (L^n, X^n)| \geq \varepsilon) = 0, \quad \varepsilon > 0.
		\end{align}
	\end{enumerate}
	Then, \(P\) solves the MP \((\Y)\), i.e. \(P \in \cM(\Y)\).
	
\end{theorem}
\begin{proof}
	It is not hard to see that the proof of Theorem \ref{theo: main1 pre} remains valid in case the following two properties hold:
	\begin{align}\label{eq: to show3}
	&\{Y^\circ_s (L, X) \colon s \in D \cap [0, t]\} \text{ is uniformly integrable for all \(t \in D\),} 
	\\\label{eq: to show4}
	&E^{P^n} \big[ Y^n_t Z^\circ_s (L^n, X^n) \big] \to E^P \big[ Y^\circ_t (L, X) Z^\circ_s (L, X) \big] \text{ for all } s, t \in D, s \leq t, Z^\circ_s \in \Z_s.
	\end{align}
	Indeed, \eqref{eq: to show3} suffices for \eqref{eq T362}, and if \eqref{eq: to show4} holds, we write for \(s, t \in D, s < t, Z^\circ_s \in \mathcal Z_s^\circ\)
	\begin{align*}
	E^P\big[(Y_t - Y_s)Z^\circ_s(L, X)\big] & = E^P\big[(Y_t^\circ(L, X) - Y_s^\circ(L, X))Z^\circ_s(L, X)\big] \\ & = \lim_{n\to\infty} E^{P_n}\big[(Y_t^n - Y_s^n)Z^\circ_s(L^n, X^n)\big] = 0,
	\end{align*}
	i.e.\ the conclusion of \eqref{eq T361} holds as well. 
	
	For \eqref{eq: to show3}, note that \eqref{eq: main conv cond mart approx} and Theorem \ref{coro: CMT} yield that for all $N>0$ and $t\in D$
	\begin{align*}
	\lim_{n \to \infty} & \big|E^{P^n} \big[ |Y^n_t| \wedge N \big] - E^P \big[ |Y^\circ_t (L, X)| \wedge N \big] \big| \\ & \leq \lim_{n \to \infty} E^{P^n} \big[ |Y^n_t - Y^\circ_t (L^n, X^n)|\wedge N  \big] \\&\qquad\qquad+ \lim_{n \to \infty} \big| E^{P_n} \big[ |Y^\circ_t (L^n, X^n)| \wedge N\big] - E^P\big[|Y^\circ_t (L, X)| \wedge N \big] \big| = 0.
	\end{align*}
	Hence, for every \(N > 0\) and \(t \in D\) we have 
	\begin{align*}
	E^P \big[ |Y^\circ_t (L, X)| - |Y^\circ_t (L,X)| \wedge N \big] &= \lim_{m \to \infty} E^P \big[ |Y^\circ_t (L, X)| \wedge m - |Y^\circ_t (L, X)| \wedge N \big]
	\\&= \lim_{m \to \infty} \lim_{n \to \infty}  E^{P^n} \big[ |Y^n_t| \wedge m - |Y^n_t| \wedge N \big]
	\\&\leq \sup_{n \in \mathbb{N}} E^{P^n} \big[ |Y^n_t| - |Y^n_t| \wedge N \big].
	\end{align*}
	Together with uniform integrability of $\{Y_s^n: s\in D \cap[0,t], n\in\mathbb N\}$ and Lemma \ref{lem: UI JS}, this inequality yields \eqref{eq: to show3}.
	Next, we verify \eqref{eq: to show4}. For \(s, t \in D, s \leq t, Z^\circ_s \in \Z_s\) and \(N > 0\), we obtain
	\begin{equation*} 
	\begin{split}
	\big| E^P \big[ Y^\circ_t  (L, X) &Z^\circ_s (L, X) \big] - E^{P^n} \big[ Y^n_t Z^\circ_s(X^n) \big] \big| \\&\lesssim E^P \big[ |Y^\circ_t (L,X)|- |Y_t^\circ (L,X) | \wedge \N  \big] \\ & \quad \quad + \big| E^P \big[(Y^\circ_t (L,X) \vee (- \N) \wedge \N) Z^\circ_s (L,X) \big] 
	\\ & \qquad \qquad \qquad - E^{P^n} \big[(Y^\circ_t (L^n,X^n) \vee (- \N) \wedge \N) Z^\circ_s(L^n,X^n) \big] \big|
	\\&\quad\quad + E^{P^n} \big[ \big| Y^\circ_t (L^n,X^n) \vee (-N) \wedge N - Y^n_t \vee (-N) \wedge N \big| \big]
	\\&\quad\quad+ E^{P^n} \big[ |Y^n_t | - |Y^n_t | \wedge \N | \big]
	\\&\triangleq I_1 + I_2 + I_3 + I_4.
	\end{split}
	\end{equation*}
	Theorem \ref{coro: CMT} yields that \(I_2 \to 0\) as \(n \to \infty\). Moreover, \eqref{eq: main conv cond mart approx} implies that \(I_3 \to 0\) as \(n \to \infty\). Finally, uniform integrability and Lemma \ref{lem: UI JS} yield that \(I_1 + I_4 \to 0\) as \(N \to \infty\) uniformly in \(n\). In summary, we conclude that \eqref{eq: to show4} holds and hence the proof is complete.
\end{proof}

In Theorem \ref{theo: main approx} we do not impose integrability assumptions on the elements of \(\Y\) but on its approximation sequences. Hence, Theorems \ref{theo: main1 pre} and \ref{theo: main approx} have different scopes and do not imply each other.

Finally, let us again comment on the case without control variables. The following is an immediate consequence of Theorem \ref{theo: main approx}.
\begin{corollary} \label{coro: main singleton approx}
	Let all assumptions from Corollary \ref{coro: main coro singleton} hold, except \textup{(S3)}. Suppose that for every \(s \in D\) and \(Z^\circ_s \in \Z_s\) the random variable \(Z^\circ_s (X^n)\) is \(\mathcal{F}^n_s\)-measurable, and that the following holds:
	\begin{enumerate}
		\item[\textup{(S4)}] \(\Y\) is canonical and for every \(Y \in \Y\) there exists a canonical version \((Y^\circ_t)_{t \geq 0}\) such that for every \(t \in D, s \in D \cap [0, t]\) and \(Z^\circ_s \in \Z_s\) the following hold: \(Y^\circ_t\) and 
		\(Y^\circ_t Z^\circ_s\) are \(P\)-continuous at~\(X\). Moreover, there exists a sequence \((Y^n)_{n \in \mathbb{N}}\) such that \(Y^n\) is a martingale on \(\B^n\), the set \(\{Y^n_s \colon s \in D \cap [0, t], n \in \mathbb{N}\}\) is uniformly integrable and 
		\begin{align*} 
		\lim_{n \to \infty} P^n (|Y^n_t - Y^\circ_t (X^n)| \geq \varepsilon) = 0, \quad \varepsilon > 0.
		\end{align*}
	\end{enumerate}
	Then, \(P\) solves the MP \((\Y)\), i.e. \(P \in \cM(\Y)\).
	
\end{corollary}
In the next section we relate the results above to known theorems from the literature. Thereafter, in Section \ref{sec: stab semi ftd} we present new results which are tailored to processes with fixed times of discontinuity. For these results it is crucial that we can work with the concept of weak-strong convergence.

\section{Relation to Existing Results} \label{sec: Revisited Literature main section}
The purpose of this section is to specialize the terminologies introduced in the previous section to three examples taken from the literature:  In Section \ref{sec: rec EK} we recover the classical convergence theorem for Markovian martingale problems as presented in the monograph \cite{EK} by Ethier and Kurtz, and in Section \ref{sec: rec VSDE} we prove a mild generalization of a stability result for Volterra SDEs from \cite{jaber:hal-02279033}. Finally, in Section \ref{sec: rec JS} we localize a theorem by Jacod and Shiryaev \cite{JS} for semimartingales by replacing a global with a local boundedness hypothesis on the semimartingale characteristics. Such a generalization has been announced in \cite{JS}, but was not stated in a precise manner. We believe it to be useful for future applications and therefore of independent interest.

\subsection{Relation to a Theorem by Ethier and Kurtz} \label{sec: rec EK}
Let \(E\) be a Polish space, for every \(n \in \mathbb{N}\) let \(\B =(\Omega, \mathcal{F}, \F, P)\) and \(\B^n = (\Omega^n, \mathcal{F}^n, \F^n, P^n)\) be filtered probability spaces which support \(E\)-valued \cadlag adapted processes \(X\) and \(X^n\), respectively. Moreover, suppose that the filtration \(\F\) on $\B$ is generated by \(X\).
Let \(A \subset C_b (E) \times C_b (E)\) and define \(\Y\) to be the set of the following processes:
\begin{align} \label{eq: main EK test processes}
f (X) - f(X_0) - \int_0^\cdot g (X_s) ds, \quad (f, g) \in A.
\end{align}
Moreover, for every \(n \in \mathbb{N}\) let \(\Y^n\) be a set of pairs \((\xi, \phi)\) consisting of real-valued progressively measurable processes on \(\B^n\) such that 
\[
\sup_{s \leq T} E^{P^n} \big[ |\xi_s| + |\phi_s| \big] < \infty, \quad T > 0,
\]
and such that
\[
\xi - \int_0^\cdot \phi_s ds 
\]
is a martingale on $\mathbb B^n$.
The following theorem is a  version of the implication (c\('\)) $\Rightarrow$ (a\('\)) from \cite[Theorem 4.8.10]{EK}.
\begin{theorem} \label{theo: EK restate}
	Suppose that \(X^n \to X\) weakly on \(\D(E)\) endowed with the Skorokhod \(J_1\) topology, and that there exists a set \(\Gamma \subset \mathbb{R}_+\) with countable complement such that for each \((f, g) \in A\) and \(T > 0\), there exists a sequence \((\xi^n, \phi^n) \in \Y^n\) such that 
	\begin{align}\label{eq: EK UI1}
	\sup_{n \in \mathbb{N}} \sup_{s \leq T} E^{P^n} \big[ |\xi^n_s| + |\phi^n_s| \big] < \infty,
	\end{align}
	\begin{align} \label{eq: EK UI2}
	\lim_{n \to \infty} E^{P^n} \Big[ (\xi^n_t - f(X^n_t)) \prod_{i = 1}^k h_i (X^n_{t_i}) \Big] = 0,
	\end{align}
	\begin{align} \label{eq: EK UI3}
	\lim_{n \to \infty} E^{P^n} \Big[ \int_s^t (\phi^n_u - g(X^n_u)) du \prod_{i = 1}^k h_i (X^n_{t_i}) \Big] = 0,
	\end{align}
	for all \(k \in \mathbb{N}, t_1, \dots, t_k \in \Gamma \cap [0, t], t \in \Gamma \cap [0, T]\), \(h_1, \dots, h_k \in C_b (E)\).
	Then, \(P\in \cM(\Y)\).
\end{theorem}
\begin{proof}
	We check (S1) -- (S3) in Corollary \ref{coro: main coro singleton}. Of course, (S1) holds by hypothesis. 
	Let \(\Z = \{\Z_t, t \in \mathbb{R}_+\}\) be as in part (i) of Example \ref{ex: determining set} with
	\[
	D \triangleq \big\{ t \in \Gamma \colon P (X_t \not = X_{t-}) = 0 \big\}.
	\]
	Since \(\Gamma^c\) is countable, \(D^c\) is countable (see \cite[Lemma 3.7.7]{EK}) and consequently, \(D\) is dense in \(\mathbb{R}_+\). As explained in Example \ref{ex: determining set}, \(\Z\) is a determining set for \(\Y\). 
	Thus, (S2) holds, too. Finally, we check (S3).  It is clear that \(\Y\) is canonical and that every \(Y^\circ_t\) is bounded, which implies that \(\{Y^\circ_r (X^n) \colon r \in D \cap [0, t], n \in \mathbb{N}\}\) is uniformly integrable.
	Moreover, as \(\omega \mapsto \omega (t)\) is continuous at \(\omega\) whenever \(\omega (t) = \omega(t-)\), for every \(t \in D\) any \(Z^\circ_t \in \Z_t\) is \(P\)-a.s. continuous at \(X\) by definition of \(D\).
	Similarly, again by definition of \(D\), for every \(t \in D\) the random variable \(Y^\circ_t\) is \(P\)-a.s. continuous at \(X\). It is left to verify the final part of (S3). Take \(Y \in \Y\) such that 
	\[
	Y = f (X) - f(X_0) - \int_0^\cdot g(X_s) ds,
	\]
	and set 
	\[
	Y^n \triangleq \xi^n - f(X^n_0) - \int_0^\cdot \phi^n_s ds,  
	\]
	where \(\xi^n\) and \(\phi^n\) are as in \eqref{eq: EK UI1}, \eqref{eq: EK UI2} and \eqref{eq: EK UI3}. Let \(s, t \in D \subset \Gamma\) with \(s < t\) and take \(Z^\circ_s \in \Z_s\). Clearly, we have 
	\begin{align*}
	(Y^\circ_t (X^n) - Y^n_t &+ Y^n_s - Y^\circ_s (X^n)) Z^\circ_s (X^n) \\&= (f(X^n_t) - \xi^n_t + \xi^n_s - f (X^n_s)) \prod_{i = 1}^k h_i(X^n_{t_i}) 
	\\&\hspace{2cm}+ \int_s^t(\phi^n_u - g(X^n_u)) du \prod_{i = 1}^k h_i (X^n_{t_i})
	\end{align*}
	for certain \(k \in \mathbb{N}, t_1, \dots, t_k \in \Gamma \cap [0, s], h_1, \dots, h_k \in C_b(E)\) related to \(Z^\circ_s\). 
	The \(P^n\)-expectation of the first term converges to zero by \eqref{eq: EK UI2}, and the \(P^n\)-expectation of the second term converges to zero by \eqref{eq: EK UI3}. 
	As \(Y^n\) is a martingale on \(\B^n\) we have
	\[
	E^{P^n} \big[ (Y^n_s - Y^n_t) Z^\circ_s(X^n)\big] = 0,
	\]
	and consequently,
	\begin{align*}
	\lim_{n \to \infty} E^{P^n} &\big[ (Y^\circ_t (X^n) - Y^\circ_s (X^n)) Z^\circ_s(X^n) \big] \\&= \lim_{n \to \infty} E^{P^n} \big[ (Y^\circ_t (X^n) - Y^n_t + Y^n_s - Y^\circ_s (X^n)) Z^\circ_s(X^n)\big] = 0.
	\end{align*}
	We conclude that (S3) holds. Hence, the claim follows from Corollary~\ref{coro: main coro singleton}.
\end{proof}
\begin{remark}
	In \cite[Theorem 4.8.10, (a\('\)) $\Rightarrow$ (c\('\))]{EK} it is shown that in case \(X^n \to X\) and \(P \in \cM (\Y)\), there exist processes \((\xi^n, \phi^n) \in \Y^n\) with the properties \eqref{eq: EK UI1}, \eqref{eq: EK UI2} and~\eqref{eq: EK UI3}. 
\end{remark}

In Section \ref{sec: EK FTD} below we derive a version of Theorem \ref{theo: EK restate} where the Lebesgue measure in \eqref{eq: main EK test processes} is replaced by a general locally finite measure which is allowed to have point masses. At this point we stress that the proof (and the result itself) requires substantial adjustments, as in this case the test processes have no Skorokhod \(J_1\) continuous canonical versions in general. More comments on this issue are given at the end of Section~\ref{sec: rec JS}.

\subsection{A Stability Result for Volterra Equations} \label{sec: rec VSDE}
In this section we discuss a stability result for Volterra SDEs (VSDEs) of the type
\begin{align} \label{eq: VSDE}
X_t = g_0 (t) + \int_0^t K_{t - s} d Z_s, \quad t \in \mathbb{R}_+,
\end{align}
where $X$ is an $\mathbb R^d$-valued predictable process and  \(Z\) is an $\mathbb R^k$-valued semimartingale with differential characteristics \((b (X), a(X), \nu (X))\), i.e. with semimartingale characteristics \((B^Z, C^Z, \nu^Z)\) of the form
\[
B^Z = \int_0^\cdot b(X_s) ds, \quad C^Z = \int_0^\cdot a (X_s) ds, \quad \nu^Z(dx, dt) = \nu (X_t, dx) dt,
\]
which we suppose to correspond to a fixed continuous truncation function \(h \colon \mathbb{R}^k \to \mathbb{R}^k\).
A version of Theorem \ref{theo: stability VSDE} below has recently been proven in \cite[Theorem 3.4]{jaber:hal-02279033}. The purpose of this section is to illustrate an application of Corollary \ref{coro: main coro singleton} beyond the classical continuous or \cadlag setting.

We now provide a precise definition for solutions to VSDEs and introduce its parameters. The space \(L^p_\textup{loc} (\mathbb{R}_+, \mathbb{R}^d)\), for \(p \geq 1\) and \(d \in \mathbb{N}\) endowed with the local \(L^p\)-norm topology, will serve as state space of the process \(X\) in \eqref{eq: VSDE}. Now, we introduce the following coefficients:
\begin{enumerate}
	\item[(D1)] An initial value \(g_0 \in L^p_\textup{loc} (\mathbb{R}_+, \mathbb{R}^d)\).
	\item[(D2)] A convolution kernel \(K \colon \mathbb{R}_+ \to \mathbb{R}^{d \times k}\) in \(L^p_\textup{loc} (\mathbb{R}_+, \mathbb{R}^d)\).
	\item[(D3)] A characteristic triplet \((b, a, \nu)\) consisting of two Borel functions \(b \colon \mathbb{R}^d \to \mathbb{R}^k\) and \(a \colon \mathbb{R}^d \to \mathbb{S}^k_+\) and a Borel transition kernel \(\nu\) from \(\mathbb{R}^d\) into \(\mathbb{R}^k\) which does not charge the set \(\{0\}\). Furthermore, we suppose that there exists a constant \(c > 0\) such that for all \(x \in \mathbb{R}^d\)
	\[
	\|b(x)\| + \|a(x)\| + \int (1\wedge\|y\|^2) \nu(x, dy) \leq c (1 + \|x\|^p). %, \quad x \in \mathbb{R}^d.
	\]
\end{enumerate}
We are in the position to define solutions to the VSDE \eqref{eq: VSDE}.
\begin{definition}
	A triplet \((\B, X, Z)\) is called a \emph{weak solution} to the \emph{Volterra SDE (VSDE)} associated to \((g_0, K, b, a, \nu)\), if \(\B\) is a stochastic basis which supports two processes \(X\) and \(Z\), where \(X\) is \(\mathbb{R}^d\)-valued, predictable and has paths in \(L^p_\textup{loc} (\mathbb{R}_+, \mathbb{R}^d)\), \(Z\) is an \(\mathbb{R}^k\)-valued \cadlag semimartingale with differential characteristics \((b(X), a(X), \nu(X))\), and \eqref{eq: VSDE} holds.
\end{definition}

Let \((g^n_0, K^n, b^n, a^n, \nu^n)\) and \((g_0, K, b, a, \nu)\) be coefficients for VSDEs. 
Moreover, for every \(f \in C^2_c (\mathbb{R}^k)\) and \((x, z) \in \mathbb{R}^d \times \mathbb{R}^k\) we set 
\begin{align*}
\mathcal{L} f (x, z) \triangleq \langle b(x)&, \nabla f (z)\rangle + \tfrac{1}{2}\operatorname{tr} ( a(x) \nabla^2 f(z))
\\&+ \int \big( f (z + y) - f(z) - \langle h (y), \nabla f (z)\rangle \big) \nu(x, dy).
\end{align*}
Similar to \(\mathcal{L}\), we define \(\mathcal{L}^n\) with \((b, a, \nu)\) replaced by \((b^n, a^n, \nu^n)\). Set \(F = L^p_\textup{loc} (\mathbb{R}_+, \mathbb{R}^d) \times \D(\mathbb{R}^k)\) endowed with the product topology, where \(\D(\mathbb{R}^k)\) is endowed with the Skorokhod \(J_1\) topology.

\begin{theorem} \label{theo: stability VSDE}
	Let \((\B^n, X^n, Z^n)\) be a weak solution to the VSDE \((g^n_0, K^n, b^n, a^n, \nu^n)\) for every \(n \in \mathbb{N}\), and let \((\Omega, \mathcal{F}, P)\) be a probability space which supports a measurable process \((X, Z)\) with paths in~\(F\). Set \(\mathcal{F}_t \triangleq \sigma (X_s, Z_s, s \leq t)\) for \(t \in \mathbb{R}_+\) and \(\B \triangleq (\Omega, \mathcal{F}, \bF \triangleq (\cF^P_{t+})_{t \geq 0}, P)\). Suppose that \(X\) is \(\F\)-progressively measurable, that \eqref{eq: X = liminf} holds and assume the following: 
	\begin{enumerate}
		\item[\textup{(L1)}] \((X^n,Z^n) \to (X, Z)\) weakly in \(F\).
		\item[\textup{(L2)}] For every \(f \in C^2_c (\mathbb{R}^k)\) there exists a constant \(c_f > 0\) such that 
		\[
		| \mathcal{L}^n f (x, z) | \leq c_f (1 + \|x\|^p ), \quad (n, x, z) \in \mathbb{N} \times \mathbb{R}^d \times \mathbb{R}^k.
		\]
		\item[\textup{(L3)}] \(\mathcal{L} f\) is continuous for every \(f \in C^2_c (\mathbb{R}^k)\).
		\item[\textup{(L4)}] \(g^n_0 \to g_0\) and \(K^n \to K\) in \(L^p_\textup{loc}(\mathbb{R}_+, \mathbb{R}^d)\), and 
		\begin{equation} \label{eq: last part L4}
		\begin{split}
		f \in C^2_c(\mathbb{R}^k),\ \mathbb{R}^{d + k} \ni (x_n&, z_n) \to (x, z) \in \mathbb{R}^{d + k} \\&\Longrightarrow \quad \big|\mathcal{L}^n f (x_n, z_n) - \mathcal{L} f (x_n, z_n) \big| \to 0.
		\end{split}
		\end{equation}
	\end{enumerate}
	Then, \((\B, X, Z)\) is a weak solution to the VSDE \((g_0, K, b, a, \nu)\).
\end{theorem}
\begin{remark}
	\begin{enumerate}
		\item[\textup{(i)}]
		The assumption that \(X\) is \(\F\)-progressively measurable comes without loss of generality, see \cite[Proposition 1.1.12]{KaraShre}.
		\item[\textup{(ii)}]
		It is clear that \eqref{eq: last part L4} holds whenever \(\mathcal{L}^n f \to \mathcal{L} f\) locally uniformly for every \(f \in C^2_c (\mathbb{R}^k)\). This together with the continuity of each \(\mathcal{L}^n f\) is used in \cite{jaber:hal-02279033}.\footnote{In the statement of \cite[Theorems 1.6, 3.4]{jaber:hal-02279033} the coefficients are not assumed to be continuous, but this assumption is used in the proof, see \cite[Lemma 3.6]{jaber:hal-02279033}.} Of course, this already implies (L3). In Theorem \ref{theo: stability VSDE} we only ask for the weaker assumptions (L3) and \eqref{eq: last part L4}. Hence, Theorem \ref{theo: stability VSDE} can also be applied when \(\mathcal{L}^n\) has discontinuous coefficients.
	\end{enumerate}
\end{remark}
\begin{proof}
	Using standard extensions, we can assume that \(\B^n\) and \(\B\) support random variables \(U^n\) and \(U\) which are uniformly distributed on \([0, 1]\) such that the following hold: \(U^n\) is independent of \((X^n, Z^n)\), \(U\) is independent of \((X, Z)\) and \((U^n, X^n, Z^n) \to (U, X, Z)\) weakly as \(n \to \infty\). Furthermore, we can redefine \(\mathcal{F}^n_0 \triangleq \mathcal{F}^n_0 \vee \sigma (U^n)\) and \(\mathcal{F}_0 \triangleq \mathcal{F}_0 \vee \sigma (U)\).
	Let \((\U, \X, \z)\) be the identity map on \([0, 1] \times F\) and~set
	\[
	T_m \triangleq \inf \Big( t \in \mathbb{R}_+ \colon \int_0^t (1 + \U + \|\X_s\|^p) ds \geq m\Big), \quad m > 0.
	\]
	Moreover, let \(\Y^\circ\) be the set of all processes
	\[
	f (\z_{\cdot \wedge T_m}) - \int_0^{\cdot \wedge T_m} \mathcal{L} f (\X_s, \z_s) ds, \quad m > 0, f \in C_c^2(\mathbb{R}^k),
	\]
	and define \(\Y^n\) to be the set of all processes
	\[
	f (Z^n_{\cdot \wedge T_m (U^n, X^n)}) - \int_0^{\cdot \wedge T_m (U^n, X^n)} \mathcal{L} f (X^n_s, Z^n_s) ds, \quad m > 0, f \in C_c^2(\mathbb{R}^k).
	\]
	The following lemma is a direct consequence of  \cite[Lemma 3.3]{jaber:hal-02279033}.
	\begin{lemma} \label{lem: mc VSDE}
		\begin{enumerate}
			\item[\textup{(i)}] \((\B, X, Z)\) is a weak solution to the VSDE \((g_0, K, b, a, \nu)\), if for every \(Y^\circ \in \Y^\circ\) the process \(Y^\circ (U, X, Z)\) is a martingale, and 
			\begin{align} \label{eq: second part mc VSDE}
			\int_0^t X_s ds = \int_0^t g_0 (s) ds + \int_0^t K_{t-s} Z_s ds, \quad t \in \mathbb{R}_+. 
			\end{align}
			\item[\textup{(ii)}] All processes in \(\Y^n\) are martingales on \(\B^n\).
		\end{enumerate}
	\end{lemma}
	Equation \eqref{eq: second part mc VSDE} follows from (L1) and (L4), see \cite[Lemma 3.5]{jaber:hal-02279033} for details. Thus, to conclude the claim of the theorem, it suffices to show that all processes in \(\Y \triangleq \{Y^\circ (U, X, Z) \colon Y^\circ \in \Y^\circ\}\) are martingales. To show this we use Corollary \ref{coro: main coro singleton}.

	First of all, note that (S1) in Corollary \ref{coro: main coro singleton} coincides with (L1). Thus, we only need to verify (S2) and (S3), where we take
	\[
	D \triangleq \big\{t > 0 \colon P (\Delta Z_t \not = 0) = 0 \big\}.
	\]
	It is well-known that \(D\) is dense in \(\mathbb{R}_+\) (see \cite[Lemma 3.7.7]{EK}). 
	
	For \(t \in \mathbb{R}_+\), define \(\Z_t\) to be the set of functions 
	\[
	\prod_{i = 1}^m h_i \Big(\U, \int_0^{t_i} \X_s ds, \z_{t_i} \Big), 
	\]
	where \(m \in \mathbb{N}, t_1, \dots, t_m \in D \cap [0, t]\) and \(h_1, \dots, h_m \in C_b (\mathbb{R}^{1 + d + k})\). Recalling both parts of Example \ref{ex: determining set}, we note that \(\Z = \{\Z_t, t \in \mathbb{R}_+\}\) is a determining set for \(\Y\). Thus, (S2) holds.
	
	Finally, we verify (S3). 
	It is obvious by definition that \(\Y^\circ\) is canonical and, by (D3) and the definition of \(T_m\), all processes in \(\Y^\circ\) are bounded on finite time intervals. 
	Furthermore, by the definition of \(D\), for every \(t \in \mathbb{R}_+\) all elements of \(\Z_t\) are \(P\)-a.s. continuous at \((U, X, Z)\).
	Take \(t \in D\) and \(Y^\circ \in \Y^\circ\) such that
	\[
	Y^\circ_t = f (\z_{t \wedge T_m}) - \int_0^{t \wedge T_m} \mathcal{L} f (\X_s, \z_s) ds.
	\]
	Note that
	\begin{align*}
	\{T_m > s\} &= \Big\{ \int_0^s (1 + \U +\|\X_u\|^p) du < m \Big\}, \\
	\{T_m < s\} &= \Big\{ \int_0^s (1 + \U+ \|\X_u\|^p) du > m \Big\},
	\end{align*}
	because \(s \mapsto \int_0^s (1 + \U + \|\X_u\|^p)du\) is strictly increasing.
	Hence, using the continuity of \((\U, \X) \mapsto \int_0^s (1 + \U + \|\X_u\|^p)du\), the map \(T_m\) is upper and lower semicontinuous and consequently, continuous. 
	With this observation at hand, it follows easily from the continuity of \(\mathcal{L} f\) that \[(\U, \X, \z) \mapsto \int_0^{t \wedge T_m (\U, \X)} \mathcal{L} f (\X_s, \z_s) ds\] is continuous, too. 
	Thanks to the randomization given by \(U\), a.s. \(Z\) does not jump at time \(T_m (U, X)\), see \cite[p. 17]{jaber:hal-02279033} for details. 
	Thus, \((\U, \X, \z) \mapsto \z_{t \wedge T_m (\U, \X)}\) is \(P\)-continuous at \((U, X, Z)\) for every \(t \in D\), by the definition of the set \(D\) and \cite[Proposition VI.2.1]{JS}.
	In summary, \(Y^\circ_t\) and \(Y^\circ_t Z^\circ_s\) for \(Z^\circ_s \in \Z_s\) and \(s \leq t\) are \(P\)-continuous at \((U, X, Z)\) for every~\(t \in D\).

	It is left to verify the final part of (S3).
	Let \(Y^n \in \Y^n\) be given by
	\[
	f (Z^n_{\cdot \wedge T_m (U^n, X^n)}) - \int_0^{\cdot \wedge T_m (U^n, X^n)} \mathcal{L}^n f (X^n_s, Z^n_s) ds.
	\]
	Thanks to part (ii) of Lemma \ref{lem: mc VSDE}, \(Y^n\) is a martingale on \(\B^n\). 
	By Skorokhod's coupling theorem (\cite[Theorem 3.30]{kallenberg}), we can and will assume that \((U^n, X^n, Z^n)\) and \((U, X, Z)\) are defined on the same probability space and that a.s. \((U^n, X^n, Z^n) \to (U, X, Z)\) as \(n \to \infty\).
	We now show that 
	\begin{align} \label{conv VSDE to show}
	E \big[ \big|Y^n_t - Y^\circ_t (U^n, X^n, Z^n)\big| \big] \to 0 \text{ as } n \to \infty,
	\end{align}
	i.e. \eqref{eq: L1 cond} in Proposition \ref{prop: mg approx}. This implies the last part in (S3) and thereby completes the proof.
	We see that 
	\[
	Y^n_t - Y^\circ_t (U^n, X^n, Z^n) = \int_0^{t \wedge T_m (U^n, X^n)} (\mathcal{L}^n f (X^n_s, Z^n_s) - \mathcal{L} f (X^n_s, Z^n_s)) ds.
	\]
	The implication \eqref{eq: last part L4} yields that a.s. for a.a. \(s \in [0, t]\)
	\[
	\big|\mathcal{L}^n f (X^n_s, Z^n_s) - \mathcal{L} f (X^n_s, Z^n_s)\big|  \to 0 \text{ as } n \to \infty.
	\]
	As a.s. \(X^n \to X\) in \(L^p_\textup{loc}(\mathbb{R}_+, \mathbb{R}^k)\), the family \(\{\|X^n\|^p\colon n \in \mathbb{N}\}\) restricted to any finite time interval is a.s. uniformly integrable w.r.t. the Lebesgue measure.
	Thus, using (D3) and (L2), we deduce from Vitali's theorem that a.s.
	\[
	\int_0^{t} \big| \mathcal{L}^n f (X^n_s, Z^n_s) - \mathcal{L} f (X^n_s, Z^n_s) \big| ds \to 0 \text{ as } n \to \infty.
	\]
	Finally, since
	\[
	\int_0^{t \wedge T_m (U^n, X^n)} \big| \mathcal{L}^n f (X^n_s, Z^n_s) - \mathcal{L} f (X^n_s, Z^n_s) \big| ds \lesssim 1 + m
	\]
	by (D3), (L2) and the definition of \(T_m\), the dominated convergence theorem yields \eqref{conv VSDE to show}. The proof is complete.
\end{proof}

\subsection{An Extension of a Theorem by Jacod and Shiryaev} \label{sec: rec JS}
Let \((B, C, \nu)\) be a candidate triplet for semimartingale characteristics defined on the canonical space \(\D(\mathbb{R}^d)\) endowed with the Skorokhod \(J_1\) topology, see \cite[III.2.3]{JS} for the technical requirements. Here, we assume that \((B, C, \nu)\) corresponds to a continuous truncation function \(h \colon \mathbb{R}^d \to \mathbb{R}^d\).

For \(m > 0\) we set 
\begin{equation} \label{eq: def T, Theta} \begin{split}
T_m (\omega) &\triangleq \inf (t \in \mathbb{R}_+ \colon \|\omega (t)\| \geq m \text{ or } \|\omega (t-)\| \geq m),\quad \omega \in \D(\mathbb{R}^d),
\end{split}
\end{equation}
\begin{align*}
\Theta_{m, t} &\triangleq \Big\{ \omega \in \D(\mathbb{R}^d) \colon \sup_{s \leq t} \|\omega(s)\| \leq m\Big\}.
\end{align*}
Let \(C_1 (\mathbb{R}^d)\) be a subset of the set of non-negative bounded continuous functions vanishing around the origin as described in \cite[VII.2.7]{JS}.

The following theorem generalizes \cite[Theorem IX.2.11]{JS} for the quasi-left continuous case as outlined on p. 533 in \cite{JS}.
\begin{theorem}\label{theo: JS generalization}
	Let \((\Omega, \mathcal{F}, \bF, P)\) and \((\Omega^n, \cF^n, \bF^n, P^n)\) be filtered probability spaces which support \(\mathbb{R}^d\)-valued c\`adl\`ag adapted processes \(X\) and \(X^n\) such that each \(X^n\) is a semimartingale with semimartingale characteristics \((B^n, C^n, \nu^n)\) corresponding to the (continuous) truncation function~\(h\). Assume that \(X^n \to X\) weakly on \(\D(\mathbb{R}^d)\) and that the following hold:
	\begin{enumerate}
		\item[\textup{(i)}] There exists a set \(\Gamma \subset \mathbb{R}_+\) with countable complement such that for every \(t \in \Gamma\), \(m, \varepsilon > 0\) and \(g \in C_1(\mathbb{R}^d)\) we have 
		\begin{align*}
		P^n (\|B^n_{t \wedge T_m (X^n)}  - B_{t \wedge T_m (X^n)} (X^n)\| \geq \varepsilon) &\to 0,\\ P^n (\|\widetilde{C}^n_{t \wedge T_m (X^n)}  - \widetilde{C}_{t \wedge T_m (X^n)} (X^n)\| \geq \varepsilon) &\to 0,\\ P^n(|g * \nu^n_{t \wedge T_m (X^n)}  - g * \nu_{t \wedge T_m (X^n)} (X^n)| \geq \varepsilon) &\to 0,
		\end{align*}
		as \(n \to \infty.\)
		\item[\textup{(ii)}]
		For all \(t \in \mathbb{R}_+, m > 0\) and \(g \in C_1(\mathbb{R}^d)\) we have 
		\begin{align}\label{eq: bounded JS gen}
		\sup_{\omega \in \Theta_{m, t}} \big(\| \widetilde{C}_t (\omega)\| + | (g * \nu_t)(\omega) | \big) < \infty.
		\end{align}
		\item[\textup{(iii)}]	There exists a dense set \(\Gamma^* \subset \mathbb{R}_+\) such that for all \(t \in \Gamma^*\) and \(g \in C_1(\mathbb{R}^d)\) the maps 
		\[
		\D(\mathbb{R}^d) \ni \omega  \mapsto B_t (\omega), \widetilde{C}_t (\omega), (g * \nu_t) (\omega)
		\]
		are Skorokhod \(J_1\) continuous.
		\item[\textup{(iv)}]
		For every \(m > 0\) there exists a continuous increasing function \(F^m \colon \mathbb{R}_+ \to \mathbb{R}_+\) such that the processes
		\[
		F^m - \sum_{i = 1}^d \operatorname{Var} (B^{(i)}_{\cdot \wedge T_m}); \quad F^m - \Big(\sum_{i = 1}^d C^{(ii)}_{\cdot \wedge T_m} + (\|x\|^2 \wedge 1) * \nu_{\cdot \wedge T_m} \Big)
		\]
		are increasing.
	\end{enumerate}
	Then, \(X\) is a semimartingale with semimartingale characteristics \((B(X), C(X), \nu(X))\).
\end{theorem}
\begin{proof} 
	We deduce the result from Corollary \ref{coro: main singleton approx}, applied with a localized version of \(\Y\) as defined in Example \ref{ex: SMP}, see (i) -- (iii) in Example \ref{ex: SMP}. In the following we will verify (S4) in Corollary \ref{coro: main singleton approx} for a localized version of the processes in (i). The argument for the processes from (ii) can be found in the proof of Theorem \ref{theo: gen JS local uniform topology 2} below. For the processes in (iii) the argument is similar as for those in (i), see the proof of \cite[Theorem IX.2.11]{JS} for some details.
	
	Due to \cite[Propositions VI.2.11, VI.2.12]{JS} and the arguments in the proof of \cite[Proposition~IX.1.17]{JS}, there exists an increasing sequence \((k_n)_{n \in \mathbb{N}} \subset \mathbb{R}_+\) with \(k_n \to \infty\) such that the maps
	\[
	\D(\mathbb{R}^d) \ni \omega \mapsto T^n (\omega) \triangleq T_{k_n} (\omega), \omega (\cdot \wedge T^n (\omega))
	\]
	are \(P\)-a.s. Skorokhod \(J_1\) continuous at \(X\). 
	Let \(\Gamma\) be as in (i) and set 
	\begin{align}\label{eq: D JS gen}
	D \triangleq \big\{t \in \Gamma \colon P(\Delta X^{T^m (X)}_t \not = 0) = 0 \text{ for all } m \in \mathbb{N}\big\}, \quad X^{T^m (X)} \triangleq X_{\cdot \wedge T^m (X)}.
	\end{align}
	As \(D^c\) is countable, \(D\) is dense in \(\mathbb{R}_+.\)

	Fix \(m \in \mathbb{N}, T \in D\) and let \(K = K(m, T) > 0\) be such that 
	\begin{align}\label{eq: loc bound in pr}
	\sup_{\omega \in \Theta_{k_m, T}} \|\widetilde{C}_T (\omega)\| \leq K,
	\end{align}
	see hypothesis (ii).
	We define 
	\[
	S^n \triangleq \inf (t \in \mathbb{R}_+ \colon \|\widetilde{C}^n_{t \wedge T^m (X^n)}\| \geq K + 1).
	\]
	Let us recall that 
	\[
	\omega (h) \triangleq \omega - \sum_{s \leq \cdot} (\Delta \omega (s) - h (\Delta \omega (s))), \quad \omega \in \D(\mathbb{R}^d), 
	\]
	where \(h\) is the continuous truncation function we have fixed in the beginning of this section.
	We take
	\begin{align} \label{eq: Y JS theo}
	Y^\circ (\omega) \triangleq \omega (h)_{\cdot \wedge T^m (\omega) \wedge T} - \omega (0) - B_{\cdot \wedge T^m (\omega) \wedge T} (\omega),
	\end{align}
	and 
	\[
	Y^n \triangleq X^n (h)_{\cdot \wedge T^m (X^n) \wedge S^n \wedge T} - X^n_0 - B^n_{\cdot \wedge T^m (X^n) \wedge S^n \wedge T}.
	\]
	Recalling \eqref{eq: loc bound in pr}, thanks to hypothesis (i) we obtain
	\begin{align*}
	P^n (S^n \leq T) &= P^n(\|\widetilde{C}^n_{T \wedge T^m (X^n)}\| \geq K + 1) \\&\leq P^n(\|\widetilde{C}^n_{T \wedge T^m (X^n)} - \widetilde{C}_{T \wedge T^m(X^n)} (X^n)\| \geq 1) \to 0
	\end{align*}
	as \(n \to \infty\). Thus, using (i) again, for all \(t \in \Gamma\) and \(\varepsilon > 0\) we also get
	\begin{align*}
	P^n (\|Y^\circ_t (X^n) - Y^n_t\| \geq \varepsilon) \leq P^n(\|B^n_{t \wedge T \wedge T_m (X^n)} - B_{t \wedge T \wedge T_m(X^n)}&(X^n)\| \geq \varepsilon) \\&+ P^n(S^n \leq T)
	\to 0
	\end{align*}
	as \(n \to \infty\).
	Due to hypothesis (iii) and (iv) and the \(P\)-a.s. Skorokhod \(J_1\) continuity of \(\omega \mapsto T^m(\omega)\) at \(X\), it follows from \cite[IX.3.42]{JS} that for every \(t \in \mathbb{R}_+\) the map \(\omega \mapsto B_{t \wedge T^m (\omega)} (\omega)\) is \(P\)-a.s. Skorokhod \(J_1\) continuous at \(X\). 
	Moreover, whenever \(t \in D\), \cite[VI.2.3, Corollary~VI.2.8]{JS} show that the map \(\omega \mapsto \omega (h)_{t \wedge T^m (\omega)}\) is \(P\)-a.s. Skorokhod \(J_1\) continuous at \(X\).
	Consequently, also \(\omega \mapsto Y^\circ_t(\omega)\) is \(P\)-a.s. Skorokhod \(J_1\) continuous at \(X\) for every~\(t \in D\). 
	
	By the martingale problem for semimartingales (see Example \ref{ex: SMP} or \cite[Theorem~II.2.21]{JS}), \(Y^n\) is a locally square-integrable \(P^n\)-martingale whose predictable quadratic variation process is given by \(\widetilde{C}^n_{\cdot \wedge T^m (X^n) \wedge S^m \wedge T}\). Hence, it follows from Doob's inequality (\cite[Theorem I.1.43]{JS}) that for all~\(a > 0\)
	\begin{align}\label{eq: Doob}
	E^{P^n} \Big[ \sup_{s \leq a} |Y^{n, (i)}_s|^2 \Big] \leq 4 E^{P^n} \Big[ \widetilde{C}^{n,  (ii)}_{a \wedge T^m (X^n) \wedge S^m \wedge T} \Big].
	\end{align}
	As \(
	|\Delta \widetilde{C}^{n, (ij)}| \leq 2 \|h\|_\infty^2,
	\)
	the definition of \(S^m\) yields that the r.h.s. of \eqref{eq: Doob} is bounded uniformly in~\(n\). Consequently, \(\{Y^n_t \colon t \in [0, a], n \in \mathbb{N}\}\) is uniformly integrable.
	
	In summary, \(Y^\circ\) and \((Y^n)_{n \in \mathbb{N}}\) have the properties as in (S4). 
	As mentioned at the beginning of this proof, similar arguments work for suitably localized versions of the processes defined in (ii) and (iii) of Example \ref{ex: SMP}. We omit the remaining details.
\end{proof}

The third part of hypothesis (i) plus hypothesis (iv) yield quasi-left continuity of the limit, see the proof of \cite[Theorem IX.3.21]{JS}. In \cite[Theorem IX.2.11]{JS} this assumption is not needed, but it is assumed that part (iii) holds \(P\)-a.s. at \(X\). Although the dependence on \(P\) is sort of minimal, one can only benefit from it when the limit is more or less known, see \cite[Remark IX.2.13]{JS}. The monograph \cite{JS} suggests two deterministic versions: A version of (iii) (\cite[IX.2.14]{JS}) and \cite[IX.2.16]{JS}, i.e. continuity of \(\omega \mapsto B^i(\omega), \widetilde{C}^{ij}(\omega), (g * \nu) (\omega)\) from \(\D(\mathbb{R}^d)\) into \(\D(\mathbb{R})\).
As functions of the type
\begin{align}\label{eq: func ftd}
\omega \mapsto \int_0^t f (\omega (s-)) q(ds)
\end{align}
are not necessarily continuous in the Skorokhod \(J_1\) topology when \(q\) is allowed to have point masses, both of these assumptions might be too stringent for applications with fixed times of discontinuity. 
To give an example, consider \(d = 1\) and 
\[
\omega \mapsto F_t(\omega) \triangleq \int_0^t \omega (s-) \delta_1 (ds),\quad \omega \in \D(\mathbb{R}).
\]
The function \(\omega \mapsto F_t (\omega)\) is obviously continuous if \(t < 1\), but it is discontinuous for all \(t \geq 1\) as is easily seen by taking \(\omega_n = \1_{[1 - 1/n, \infty)} \to \omega = \1_{[1, \infty)}\). Thus, for this example there is no dense set \(\Gamma \subset \mathbb{R}_+\) such that \(F_t\) is continuous for all \(t \in \Gamma\). Moreover, \(\omega \mapsto F(\omega)\) is also not continuous from \(\D(\mathbb{R})\) into \(\D(\mathbb{R})\). Indeed, if \(F\) would be continuous we must have \(\omega_n \to \omega \Rightarrow F_t (\omega_n) \to F_t(\omega)\) for all \(t \not = 1\) as \(\{s > 0 \colon \Delta F_s( \omega ) \not = 0\} \subset \{1\}\), which is not true. Therefore, we note that functions of the type \eqref{eq: func ftd} do not necessarily have the continuity properties from \cite[IX.2.14, IX.2.16]{JS}.

In Section \ref{sec: FTD main JS gen} below we discuss versions of Theorem \ref{theo: JS generalization} where in (iii) the Skorokod \(J_1\) topology is replaced by the local uniform topology, which seems to us more suitable for applications to semimartingales with fixed times of discontinuity.

\section{Stability Results for Processes with Fixed Times of Discontinuity} \label{sec: stab semi ftd}

In this section we establish stability results which are tailored to processes with fixed times of discontinuity. To be more precise, in Section \ref{sec: EK FTD} we derive a version of Theorem~\ref{theo: EK restate} which applies to test processes of the type
\[
f (X) - f(X_0) - \int_0^\cdot g (X_{s-}) q(ds), 
\]
where \(q\) is a locally finite Borel measure on \(\mathbb{R}_+\) which is allowed to have point masses. Moreover, in Section \ref{sec: FTD main JS gen} we prove a version of Theorem \ref{theo: JS generalization} for semimartingales whose characteristics are only assumed to be continuous in the local uniform instead of the Skorokhod \(J_1\) topology. In both cases we introduce \emph{control variables} to work with the notion of weak-strong convergence. Before we present our results, we motivate the presence of fixed times of discontinuities.
\subsection{Motivation} \label{sec: motivation}
Continuous state branching processes (CSBP) are analogues of Galton--Watson processes in continuous time with continuous state spaces. Typically, CSBP are modeled as strong solutions to SDEs driven by a Brownian motion and a Poisson random measure. More recently, there is an increasing interest in CSBPs in random environments (CSBPRE), where the random environments are modeled by additional independent, multiplicative and (sometimes) discontinuous noise, see, e.g. \cite{10.1214/EJP.v18-2774,10.1214/EJP.v20-3812}. 
Leaving the environment random is often called the \emph{annealed} perspective. In contrast, fixing the random environment corresponds to the so-called \emph{quenched} perspective. In case the environment is represented by discontinuous noise, taking a specific path introduces fixed times of discontinuity, which therefore arise in a natural manner in the context of CSBPRE. To the best of our knowledge, the literature contains only selected stability results for quenched dynamics of CSBPRE, see, e.g.~\cite{10.1214/16-ECP6,10.1214/EJP.v20-3812}.

Fixed times of discontinuity also occur naturally in mathematical finance such as in interest rate markets in the post-crisis environment. Indeed, a closer look on historical data of European reference interest rates (see \cite[Figure 1]{Fon2020}) shows jumps at pre-scheduled dates. As a consequence, the financial literature shows an increasing interest in stochastic models for interest rates which allow for fixed times of discontinuity, see, e.g. \cite{Fon2020,10.1214/19-AAP1483}.

\subsection{A Version of the Ethier--Kurtz Theorem with Fixed Times of Discontinuity} \label{sec: EK FTD}
In this section we derive a version of Theorem \ref{theo: EK restate} which allows fixed times of discontinuity.
Let \((E, r)\) be a Polish space, let \(\B = (\Omega, \mathcal{F}, \F, P)\) and \(\B^n = (\Omega^n, \mathcal{F}^n, \F^n, P^n)\), \(n \in \mathbb{N}\), be filtered probability spaces which support \(E\)-valued \cadlag adapted processes \(X\) and \(X^n\), respectively. Moreover, suppose that the filtration \(\F\) is generated by \(X\).
Let \(D \subset C_b (E) \times C_b (E)\) and let \(q, q^1, q^2, \dots\) be  locally finite measures\footnote{A measure \(p\) on \((\mathbb{R}_+, \mathcal{B}(\mathbb{R}_+))\) is called \emph{locally finite} if \(p(K) < \infty\) for every compact set \(K \subset \mathbb{R}_+\). Equivalently, \(p\) is locally finite if \(p([0, t]) < \infty\) for all \(t > 0\).} on \((\mathbb{R}_+, \mathcal{B}(\mathbb{R}_+))\).
Let \(\Y\) be the set of the following processes:
\begin{align} \label{eq: test processes EK FTD}
f (X) - f(X_0) - \int_0^\cdot g (X_{s-}) q(ds), \quad (f, g) \in D.
\end{align}
Moreover, for every \(n \in \mathbb{N}\) let \(\Y^n\) be a set of pairs \((\xi, \phi)\) consisting of real-valued predictable processes on \(\B^n\) such that 
\[
\sup_{s \leq T} E^{P^n} \big[ |\xi_s| + |\phi_s| \big] < \infty, \quad T > 0,
\]
and such that
\[
\xi - \int_0^\cdot \phi_s q^n(ds) 
\]
is a martingale on $\mathbb B^n$.
Let \((U, \mathcal{U})\) be a measurable space and fix a \(\mathcal{U} \otimes \mathcal{B}(\mathbb{R}_+)/\mathcal{B}(\mathbb{R}_+)\) measurable function \(\u \colon U \times \mathbb{R}_+ \to \mathbb{R}_+\) such that for every \((u, t) \in U \times (0, \infty)\)
\begin{align*}
\lim_{\varepsilon \searrow 0} \sup \big\{ \u(u, s) \colon s \not = t, t - \varepsilon \leq s \leq t + \varepsilon\big\} = 0.
\end{align*}
Furthermore, let \(\k \colon \mathbb{R}_+ \to\mathbb{R}_+\) be increasing and continuous and fix a reference point \(x_0 \in E\).
Finally, we define 
\begin{align*} %\label{eq: A}
A &\triangleq \Big\{ (u, \omega) \in U \times \D(E) \colon r (\omega (t), \omega(t-)) \leq \u (u, t) \k\Big( \sup_{s \leq t}r(\omega(s), x_0)\Big)\text{ for all } t > 0 \Big\}\\&= \bigcap_{n \in \mathbb{N}} \Big\{ (u, \omega) \in U \times \D(E) \colon \\&\hspace{2cm} r ( \omega (S_n (\omega)),\omega (S_n (\omega)-)) \leq \u (u, S_n (\omega)) \kappa \Big(\sup_{s \leq S_n(\omega)}r(\omega(s), x_0)\Big) \Big\}, 
\end{align*}
where \((S_n)_{n \in \mathbb{N}}\) is an exhausting sequences for the jumps of the coordinate process. Clearly, the second line shows that \(A \in \mathcal{U} \otimes \mathcal{B}(\D(E))\).

\begin{listing} \label{ass: main1 EK FTD}
	For \(n \in \mathbb{N}\) each \(\B^n\) supports a \(U\)-valued random variable \(L^n\) such that 
	\begin{align} \label{eq: EK, conv cont}
	P^n ( (L^n, X^n) \in A) \to 1 \text{ as } n \to \infty, 
	\end{align}
	and one of the following hold:
	\begin{enumerate} \item[\textup{(a)}]
		\(P^n \circ (L^{n})^{-1} = P^{1} \circ (L^{1})^{-1}\) for all \(n \in \mathbb{N}\). % and \(m > 0\). 
		\item[\textup{(b)}]
		The \(\sigma\)-field \(\mathcal{U}\) is separable and \(\{P^n \circ (L^{n})^{-1}\colon n \in \mathbb{N}\}\) is relatively compact in \(M_m (U)\).
	\end{enumerate}
\end{listing}

\begin{listing} \label{ass: main2 EK FTD}
	There exists a dense set \(\Gamma \subset \mathbb{R}_+\) such that for each \((f, g) \in D\) and \(T > 0\), there exists a sequence \((\xi^n, \phi^n) \in \Y^n\) such that 
	\begin{align} \label{eq: EK UI1 FTD}
	\sup_{n \in \mathbb{N}} \sup_{r \leq T} E^{P^n} \big[ |\xi^n_r| + |\phi^n_r| \big] < \infty,
	\end{align}
	\begin{align}  \label{eq: EK UI2 FTD}
	\lim_{n \to \infty} E^{P^n} \Big[ (\xi^n_t - f(X^n_t)) \prod_{i = 1}^k h_i (X^n_{t_i}) \Big] = 0,
	\end{align}
	\begin{align}  \label{eq: EK UI3 FTD}
	\lim_{n \to \infty} E^{P^n} \Big[ \int_s^t(\phi^n_u - g(X^n_{u-})) q(du) \prod_{i = 1}^k h_i (X^n_{t_i}) \Big] = 0,
	\end{align}
	\begin{align} \label{eq: EK UI4 FTD}
	\lim_{n \to \infty} E^{P^n} \Big[ \int_s^t \phi^n_u (q^n - q) (du) \prod_{i = 1}^k h_i (X^n_{t_i}) \Big] = 0,
	\end{align}
	for all \(k \in \mathbb{N}, t_1, \dots, t_k \in \Gamma \cap [0, t], s, t \in \Gamma \cap [0, T], s < t\), \(h_1, \dots, h_k \in C_b (E)\).
\end{listing}

Before we state the main result of this section, let us fix the main idea behind Assumption \ref{ass: main1 EK FTD}, which is inspired by ideas from \cite{doi:10.1080/17442508108833169,JM81}. Parts (a) or (b) from Assumption \ref{ass: main1 EK FTD} ensure that we can extract a weak-strong convergence subsequence \((Q_n)_{n \in \mathbb{N}}\) from the family \(\{P \circ (L^n, X^n)^{-1} \colon n \in \mathbb{N}\}\) whose limit we denote by \(Q\). Moreover, thanks to \eqref{eq: EK, conv cont} and the partial Portmanteau theorem for weak-strong convergence (Proposition \ref{prop: limsup ws conv}), we can use the set \(A\) to verify that the test processes in \eqref{eq: test processes EK FTD} are \((Q_n, Q)\)-continuous. The trick is that for \((Q_n, Q)\)-continuity we can treat the control variables as deterministic, i.e. we only need to verify the Skorokhod \(J_1\) continuity of the time \(t\) values of the processes \eqref{eq: test processes EK FTD} on the sections \(A_u\) for each \(u \in U\). As on \(A_u\) the Skorokhod \(J_1\) and the local uniform topology coincide by Proposition \ref{prop: A local uniform Skorokhod}, this continuity follows trivially from the dominated convergence theorem. 
Roughly speaking, the paths of \(X^1, X^2, \dots\) take values in a randomized subset of the Skorokhod space \(\D(E)\) with conditionally nice topological properties which can be used thanks to the concept of weak-strong convergence. 
In general, we think that Assumption \ref{ass: main1 EK FTD} should be verifiable when the paths of \(X^1, X^2, \dots\) are constructed through the randomness of the control variables \(L^1, L^2, \dots\). This is e.g. the case for stochastic integrals, see Examples \ref{sec: semimartingale controls} and \ref{sec: RM} below for more details.
\\

The following is a version of Theorem \ref{theo: EK restate} which allows fixed times of discontinuity.
\begin{theorem} \label{theo: EK FTD}
	Suppose that \(X^n \to X\) weakly in \(\D(E)\) endowed with the Skorokhod \(J_1\) topology, and that Assumptions \ref{ass: main1 EK FTD} and \ref{ass: main2 EK FTD} hold.
	Then, \(P\in \cM(\Y)\).
\end{theorem}
\begin{proof}
	We apply Theorem \ref{theo: main1 pre} with the product space \((S, \mathcal{S}) = (U \times \D(E), \mathcal{U} \otimes \mathcal{B}(\D(E)))\).
	As in Assumption \ref{ass 2}, we set \(Q^n \triangleq P^n \circ (L^n, X^n)^{-1}\). Since (a) or (b) from Assumption~\ref{ass: main1 EK FTD} hold, and \(X^n \to X\) weakly, it follows from Theorem \ref{theo: nice version rel comp} in case (a) holds and from Theorems \ref{theo: ws conv rela comp} and \ref{theo: Mmc metrizible} in case (b) holds that 
	there exists a subsequence of \((Q_{n})_{n \in \mathbb{N}}\) which converges in \(M_{mc}(S)\) to some probability measure \(Q\).  By passing to this subsequence, we can without loss of generality assume that \((Q^n)_{n \in \mathbb{N}}\) converges in the weak-strong sense to a limit \(Q\). Consequently, (A1) holds. Clearly, we also have \(Q_{\D(E)} = P \circ X^{-1}\). 
	Let \(\Z = \{\Z_t, t \in \mathbb{R}_+\}\) be as in part (i) of Example \ref{ex: determining set} with \(D = \Gamma\), i.e. (A2) holds. It is left to show that (A3) holds. First of all, \(\Y\) is clearly canonical. Moreover, each \(Y^\circ\) is bounded on compact time intervals. Thus, also the uniform integrability assumption in (A3) holds. 
	For each \(t \in \Gamma\) and \(s \in \Gamma \cap [0, t]\) the maps \(Y^\circ_t\) and \(Y^\circ_t Z^\circ_s\) with \(Z^\circ_s \in \Z_s\) are continuous in the local uniform topology. Thus, \eqref{eq: EK, conv cont} and Propositions \ref{prop: limsup ws conv} and \ref{prop: A local uniform Skorokhod} show the \((Q^n, Q)\)-continuity assumption in (A3). Finally, it remains to show \eqref{eq: conv cond}. 
	Take \(Y \in \Y\) such that 
	\[
	Y = f (X) - f(X_0) - \int_0^\cdot g(X_{s-}) q(ds),
	\]
	and set 
	\[
	Y^n \triangleq \xi^n - f(X^n_0) - \int_0^\cdot \phi^n_s q^n(ds),  
	\]
	where \(\xi^n\) and \(\phi^n\) are as in \eqref{eq: EK UI1 FTD}, \eqref{eq: EK UI2 FTD} and \eqref{eq: EK UI3 FTD}. Let \(s, t \in \Gamma\) with \(s < t\) and take \(Z^\circ_s \in \Z_s\). We have 
	\begin{align*}
	(Y^\circ_t (X^n)- Y^n_t + Y^n_s - Y^\circ_s (X^n)) Z^\circ_s (X^n) &= (f(X^n_t) - \xi^n_t + \xi^n_s - f (X^n_s)) \prod_{i = 1}^k h_i(X^n_{t_i}) 
	\\&\hspace{1cm}+\int_s^t \phi^n_u (q^n - q)(du) \prod_{i = 1}^k h_i (X^n_{t_i})
	\\&\hspace{1cm}+ \int_s^t (\phi^n_u - g(X^n_{u-}))q(du) \prod_{i = 1}^k h_i (X^n_{t_i})
	\end{align*}
	for certain \(k \in \mathbb{N}, t_1, \dots, t_k \in \Gamma \cap [0, s], h_1, \dots, h_k \in C_b(E)\). 
	The \(P^n\)-expectation of the first term converges to zero by \eqref{eq: EK UI2 FTD}, the \(P^n\)-expectation of the second term convergences to zero by \eqref{eq: EK UI4 FTD}, and the \(P^n\)-expectation of the third term converges to zero by \eqref{eq: EK UI3 FTD}. We now can proceed as in the proof of Theorem \ref{theo: EK restate}.
	As \(Y^n\) is a martingale on \(\B^n\) we have
	\[
	E^{P^n} \big[ (Y^n_s - Y^n_t) Z^\circ_s(X^n)\big] = 0,
	\]
	and consequently,
	\begin{align*}
	\lim_{n \to \infty} E^{P^n} &\big[ (Y^\circ_t (X^n) - Y^\circ_s (X^n)) Z^\circ_s(X^n) \big] \\&= \lim_{n \to \infty} E^{P^n} \big[ (Y^\circ_t (X^n) - Y^n_t + Y^n_s - Y^\circ_s (X^n)) Z^\circ_s(X^n)\big] = 0.
	\end{align*}
	We conclude that (A3) holds. Hence, the claim follows from Theorem \ref{theo: main1 pre}.
\end{proof}

We now also consider the problem of verifying tightness of the family \(\{X^n \colon n\in \mathbb{N}\}\). 
A quite general criterion for tightness, which can be viewed as a version of Aldous' criterion for processes with fixed times of discontinuity, has recently been proved in \cite{10.1214/16-ECP6}. In the following we present an application of this tightness criterion in the spirit of \cite[Theorem 3.9.4]{EK}.
For a compact set \(K \subset E\) we~set
\[
T_K^n \triangleq \inf (t \in \mathbb{R}_+ \colon X^n_t \not \in K), \quad n \in \mathbb{N}.
\]
It is well-known that \(T^n_K\) is an \(\F^n\)-stopping time (at least when \(\F^n\) is right-continuous which we assume without loss of generality). 
We also define the stochastic interval \[\of 0, T^n_K\gs \triangleq \{(\omega, t) \in \Omega^n \times \mathbb{R}_+ \colon 0 \leq t \leq T^n_K (\omega)\}.\] We have the following tightness condition for \(\{X^n \colon n \in \mathbb{N}\}\).
\begin{theorem} \label{theo: EK FTD tight}
	Assume the following:
	\begin{enumerate}
		\item[\textup{(i)}]
		For every \(\eta, T > 0\) there exists a compact set \(K = K(\eta, T) \subset E\) such that 
		\[
		\inf_{n \in \mathbb{N}} P^n\big(X_t^n \in K  \text{ for all } 0 \leq t \leq T\big) \geq 1 - \eta.
		\]
		\item[\textup{(ii)}]
		Let \(H \subset C_b(E)\) be a subalgebra which is dense for the local uniform topology. 
		For every \(f \in H\) and any compact set \(K \subset E\) there exists a sequence \(\phi^1, \phi^2, \dots\), which of course might depend on \(f\) and \(K\), such that \((f (X^n), \phi^n) \in \Y^n\) and 
		\[
		\sup \big\{ | \phi^n_s (\omega) | \colon n \in \mathbb{N}, (\omega, s) \in \of 0, T^n_K \gs\big\} < \infty, \quad T > 0.
		\]
		\item[\textup{(iii)}] There exists an increasing \cadlag function \(Q \colon \mathbb{R}_+ \to \mathbb{R}_+\) with \(Q(0) = 0\) such that 
		\(
		q^n([0, \cdot]) \to Q
		\) as \(n \to \infty\)
		in the Skorokhod \(J_1\) topology.
	\end{enumerate}
	Then, the family \(\{X^n \colon n \in \mathbb{N}\}\) is tight (in the Skorokhod space with the Skorokhod \(J_1\) topology).
\end{theorem}
\begin{remark}
	Thanks to \cite[Theorem VI.2.15]{JS}, \(q^n  ([0, \cdot]) \to Q\) in the Skorokhod \(J_1\) topology if and only if there exists a dense set \(I \subset \mathbb{R}_+\) such that for all \(t \in I\)
	\[
	q^n([0, t]) \to Q(t), \qquad \sum_{0 < s \leq t} |q^n(\{s\})|^2 \to \sum_{0 < s\leq t} |\Delta Q (s)|^2.
	\]
	Here, we note that the l.h.s. means that \(q^n\) converges weakly to the measure induced by~\(Q\). The r.h.s. is an additional requirement.
\end{remark}
\begin{proof}[Proof of Theorem \ref{theo: EK FTD tight}]
	\emph{Step 1: Tightness of \(\{f (X^n) \colon n \in \mathbb{N}\}\).} Take \(f \in H\) and a compact set \(K \subset E\). As \(H\) is a subalgebra, \(f^2 \in H\). Denote \(\phi^1, \phi^2, \dots\) the sequence from (ii) for \(f\) and \(K,\) and let \(\psi^1, \psi^2, \dots\) be the sequence for \(f^2\) and \(K\). Fix \(0 \leq s \leq t\). We compute
	\begin{align*}
	E\big[ (f (X^n_{t \wedge T^K_n}) &- f (X^n_{s \wedge T^n_K} ))^2 | \mathcal{F}^n_s \big] \\&= E \big[ f^2 (X^n_{t \wedge T^n_K}) - 2 f (X^n_{t \wedge T^n_K}) f (X^n_{s \wedge T^n_K}) + f^2 (X^n_{s \wedge T^n_K}) | \mathcal{F}^n_s\big]
	\\&= E \Big[ f^2 (X^n_{t \wedge T^n_K}) - \int_0^{t \wedge T^n_K} \psi^n_r q^n(dr) \big| \mathcal{F}_s^n\Big] + E \Big[ \int_0^{t \wedge T^n_K} \psi^n_r q^n(dr) \big| \mathcal{F}^n_s\Big] \\&\hspace{2cm}- 2 f (X^n_{s \wedge T^n_K}) E \Big[ f(X^n_{t \wedge T^n_K}) - \int_0^{t \wedge T^n_K} \phi^n_r q^n(dr) \big| \mathcal{F}^n_s \Big] \\&\hspace{2cm}- 2 f(X^n_{s \wedge T^n_K})E \Big[ \int_0^{t \wedge T^n_K} \phi^n_r q^n(dr) \big| \mathcal{F}^n_s\Big] + f^2(X^n_{s \wedge T^n_K})
	\\&= f^2 (X^n_{s \wedge T^n_K}) - \int_0^{s \wedge T^n_K} \psi^n_r q^n(dr) + E \Big[ \int_0^{t \wedge T^n_K} \psi^n_r q^n(dr) \big| \mathcal{F}^n_s\Big] \\&\hspace{2cm}- 2 f^2 (X^n_{s \wedge T^n_K}) + 2 f(X^n_{s \wedge T^n_K}) \int_0^{s \wedge T^n_K} \phi^n_r q^n(dr) \\&\hspace{2cm}- 2 f(X^n_{s \wedge T^n_K}) E \Big[ \int_0^{t \wedge T^n_K} \phi^n_r q^n(dr) \big| \mathcal{F}^n_s\Big] + f^2(X^n_{s \wedge T^n_K})
	\\&= E \Big[ \int_{s \wedge T^n_K}^{t \wedge T^n_K} \psi^n_r q^n(dr) \big| \mathcal{F}^n_s \Big] - 2 f (X^n_{s \wedge T^n_K}) E \Big[ \int_{s \wedge T^n_K}^{t \wedge T^n_K} \phi^n_r q^n(dr) \big| \mathcal{F}^n_s \Big].
	\end{align*}
	Now, using the (local) boundedness of \(\phi^1, \phi^2,\dots\) and \(\psi^1, \psi^2, \dots\) as assumed in (ii), we obtain the existence of a constant \(C > 0\), which might depend on \(f, K\) and the sequences \(\phi^1, \phi^2, \dots\) and \(\psi^1, \psi^2, \dots\), such that 
	\begin{align*}
	E\big[ (f (X^n_{t \wedge T^K_n}) - f (X^n_{s \wedge T^n_K} ))^2 | \mathcal{F}^n_s \big] &\leq C E \Big[ \int_{s \wedge T^n_K}^{t \wedge T^n_K} q^n(dr) \big| \mathcal{F}^n_s\Big]
	\\&\leq C \int_s^t q^n(dr)
	= C \big( q^n ([0, t]) - q^n([0, s])\big).
	\end{align*}
	By virtue of assumption (iii), we conclude that the assumption A2') from \cite[Corollary~1.2]{10.1214/16-ECP6} holds.
	Next, we explain that the family \(\{f (X^n) \colon n \in \mathbb{N}\}\) also satisfies the compact containment condition given by A1) in \cite{10.1214/16-ECP6}. Take \(\eta, T> 0\) and let \(K = K(\eta, T) \subset E\) be the compact set as in assumption (i). Then, by the continuity of \(f\), the set \(f (K) \subset \mathbb{R}\) is compact. Furthermore, by (i), we obtain
	\[
	\inf_{n \in \mathbb{N}} P \big(f (X^n_t) \in f (K) \text{ for all } 0 \leq t \leq T \big) \geq \inf_{n \in \mathbb{N}}  P \big(X^n_t \in K \text{ for all } 0 \leq t \leq T \big) \geq 1 - \eta.
	\]    
	Hence, A1) from \cite{10.1214/16-ECP6} holds. Now, \cite[Corollary 1.2]{10.1214/16-ECP6} yields that \(\{f (X^n) \colon n \in \mathbb{N}\}\) is tight.
	
	~
	\emph{Step 2: Conclusion.} Due to the fact that we assume the compact containment condition (i.e. (i)) and the properties of \(H\), \cite[Theorem 3.9.1]{EK} yields that tightness of the family \(\{X^n \colon n \in \mathbb{N}\}\) is equivalent to tightness of the families \(\{f (X^n) \colon n \in \mathbb{N}\}\) for every \(f \in H\). As latter is the case thanks to Step 1, the claim of the theorem follows.
\end{proof}

\subsection{Theorems for Semimartingales} \label{sec: FTD main JS gen}
In this section we derive stability results for semimartingales which are tailored for the presence of fixed times of discontinuity. We start with a general result in Section~\ref{sec: main smg FTD}, which we specify further for the annealed case in Section \ref{sec: annealed smg FTD}, i.e. the case where all processes are defined on the same measurable space. Finally, in Section \ref{sec: ito FTD} we discuss the special case of It\^o processes with fixed times of discontinuity.
\subsubsection{The Main Result} \label{sec: main smg FTD}

We pose ourselves into the setting of Section \ref{sec: rec JS}. To be precise, let \((B, C, \nu)\) be a candidate triplet for semimartingale characteristics (corresponding to a fixed continuous truncation function \(h \colon \mathbb{R}^d \to \mathbb{R}^d\)) defined on the canonical space \(\D(\mathbb{R}^d)\). Except stated otherwise, we endow \(\D(\mathbb{R}^d)\) with the Skorokhod \(J_1\) topology.
We write \(C_1 (\mathbb{R}^d)\) for a subset of the set of non-negative bounded continuous functions vanishing in a neighborhood of the origin as described in \cite[VII.2.7]{JS}.

Let \((U, \mathcal{U})\) be a measurable space. 
We fix a \(\mathcal{U} \otimes \mathcal{B}(\mathbb{R}_+)/\mathcal{B}(\mathbb{R}_+)\) measurable function \(\u \colon U \times \mathbb{R}_+ \to \mathbb{R}_+\) such that for every \((u, t) \in U \times (0, \infty)\)
\begin{align*}
\lim_{\varepsilon \searrow 0} \sup \big\{ \u(u, s) \colon s \not = t, t - \varepsilon \leq s \leq t + \varepsilon\big\} = 0,
\end{align*}
let \(\k\colon \mathbb{R}_+ \to \mathbb{R}_+\) be increasing and continuous, and we define 
\begin{align}\label{eq: A}
A &\triangleq \Big\{ (u, \omega) \in U \times \D(\mathbb{R}^d) \colon \|\Delta \omega(t)\| \leq \u (u, t) \k \Big(\sup_{s \leq t}\|\omega(s)\|\Big)\text{ for all } t > 0 \Big\}.
\end{align}
The following is the main result of this section.
\begin{theorem} \label{theo: gen JS local uniform topology 2}
	Let \(\B = (\Omega, \mathcal{F}, \bF, P)\) and \(\B^n = (\Omega^n, \cF^n, \bF^n, P^n)\) be filtered probability spaces which support \(\mathbb{R}^d\)-valued c\`adl\`ag adapted processes \(X\) and \(X^n\), respectively, such that each \(X^n\) is a semimartingale with semimartingale characteristics \((B^n, C^n, \nu^n)\) corresponding to the (continuous) truncation function \(h\). Moreover, for each \(n \in \mathbb{N}\) let \(L^{n}\) be a \(U\)-valued random variable on \(\B^n\) such that one of the following hold:
	\begin{enumerate} \item[\textup{(a)}]
		\(P^n \circ (L^{n})^{-1} = P^{1} \circ (L^{1})^{-1}\) for all \(n \in \mathbb{N}\). 
		\item[\textup{(b)}]
		The \(\sigma\)-field \(\mathcal{U}\) is separable and \(\{P^n \circ (L^{n})^{-1}\colon n \in \mathbb{N}\}\) is relatively compact in \(M_m (U)\).
	\end{enumerate}
	Assume that \(X^n \to X\) weakly on \(\D(\mathbb{R}^d)\) and the existence of a dense set \(\Gamma \subset \mathbb{R}_+\) such that the following hold:
	\begin{enumerate}
		\item[\textup{(i)'}] For every \(t \in \Gamma, \varepsilon > 0\) and \(g \in C_1(\mathbb{R}^d)\) we have 
		\begin{align*}
		P^n (\|B^n_{t}  - B_{t} (X^n)\| \geq \varepsilon) &\to 0,\\ P^n (\|\widetilde{C}^n_{t}  - \widetilde{C}_{t} (X^n)\| \geq \varepsilon) &\to 0,\\ P^n(|g * \nu^n_{t}  - g * \nu_{t} (X^n)| \geq \varepsilon) &\to 0,
		\end{align*}
		as \(n \to \infty.\)
		\item[\textup{(ii)'}]
		For all \(T \in \Gamma\) and \(g \in C_1(\mathbb{R}^d)\) there is a sequence \(S_1, S_2, \dots\) of stopping times on \(\B^1, \B^2, \dots\), i.e. \(S_n\) is an \(\F^n\)-stopping time, such that
		\[
		P^n (S_n < T) \to 0, \quad n \to \infty,
		\]
		and
		\begin{align*}
		\sup_{n \in \mathbb{N}} E^{P^n} \Big[ \|\widetilde{C}^n_{T \wedge S_n}\|^2 + g^2 * \nu^n_{T \wedge S_n} \Big] < \infty.
		\end{align*}
		\item[\textup{(iii)'}]	For all \(t \in \Gamma\) and \(g \in C_1(\mathbb{R}^d)\) the maps 
		\[
		\D(\mathbb{R}^d) \ni \omega  \mapsto B_t (\omega), \widetilde{C}_t (\omega), (g * \nu_t) (\omega)
		\]
		are continuous in the local uniform topology. Moreover, %for all \(m > 0\) 
		\begin{align}\label{eq: all in A^n assp2}
		P^n( (L^{n}, X^n) \in A ) \to 1 \text{ as } n \to \infty.
		\end{align}
	\end{enumerate}
	Then, \(X\) is a semimartingale for its canonical filtration and its semimartingale characteristics are given by \((B(X), C(X), \nu(X))\).
\end{theorem}

\begin{proof} 
	Let \(Y^\circ\) to be any of the processes in (i) --  (iii) from Example \ref{ex: SMP}.
	We show that \(Y^\circ\) is a \(P\)-martingale for the (right-continuous) canonical filtration on \(\D(\mathbb{R}^d)\). For simplicity, we restrict our attention to the process in (ii) of Example \ref{ex: SMP}.
	More precisely, let \(Y^\circ\) be defined 
	by
	\[
	Y^\circ \triangleq V^{(i)} V^{(j)} - \widetilde{C}^{(ij)},
	\]
	where \(V = (V^{(1)}, \dots, V^{(d)})\) is given by
	\[
	V \triangleq \X (h) - \X_0 - B, 
	\]
	with
	\[
	\X (h) = \X - \sum_{s \leq \cdot} \big( \Delta \X_s - h(\Delta \X_s)\big).
	\]
	Our strategy is to apply Theorem \ref{theo: main approx}.
	We define probability measures \(Q_1, Q_2, \dots\) on the product space \((U \times \D(\mathbb{R}^d), \mathcal{U} \otimes \mathcal{B}(\D(\mathbb{R}^d)))\) via
	\[
	Q_{n} \triangleq P^n \circ \big(L^{n}, X^n\big)^{-1}, \quad n \in \mathbb{N}.
	\]
	As (a) or (b) hold, and \(X^n \to X\) weakly, by Theorem \ref{theo: nice version rel comp} in case (a) holds and by Theorems \ref{theo: ws conv rela comp} and \ref{theo: Mmc metrizible} in case (b) holds,
	there exists a subsequence of \((Q_{n})_{n \in \mathbb{N}}\) which converges in \(M_{mc}(U \times \D(\mathbb{R}^d))\) to some probability measure \(Q\). 
	To keep our notation simple, we denote the subsequence again by \((Q_{n})_{n \in \mathbb{N}}\). Clearly, we have \(Q_{\D(\mathbb{R}^d)} = P \circ X^{-1}\).
	In the following we show that \(Y^\circ_{t}\) is \((Q_{n}, Q)\)-continuous for every \(t \in \Gamma\).

	Thanks to Proposition \ref{prop: A local uniform Skorokhod}, for every \(u \in U\) the set \(A_u = \{\omega \in \D(\mathbb{R}^d) \colon (u, \omega) \in A\}\) is closed in the Skorokhod \(J_1\) topology and on \(A_u\) the Skorokhod \(J_1\) topology coincides with the local uniform topology.
	In (iii)' we assume that \(Q_{n} (A) \to 1\) as \(n \to \infty\). Hence, we deduce from Proposition \ref{prop: limsup ws conv} that \(Q(A) = 1\). The first part of assumption (iii)' yields that \(B_t|_{A_u}\) and \(\widetilde{C}_t|_{A_u}\) are continuous in the Skorokhod \(J_1\) topology for every \(t \in \Gamma\). 
	\begin{lemma}
		Let \(g \colon \mathbb{R}^d \to \mathbb{R}\) be a continuous function which vanishes in a neighborhood of the origin.
		For every \(t > 0\) the map  \(\omega \mapsto \sum_{s \leq t} g(\Delta \omega(s))\)
		is continuous in the local uniform topology.
	\end{lemma}
	\begin{proof}
		For \((\omega, u) \in \D(\mathbb{R}^d) \times (0, \infty)\) we set
		\[
		t^0 (\omega, u) \triangleq 0, \quad t^{p + 1} (\omega, u) \triangleq \inf (t > t^p (\omega, u) \colon \|\Delta \omega(t)\| > u), \quad p \in \mathbb{Z}_+.
		\]
		Furthermore, we set
		\[
		U(\omega) \triangleq \big\{u > 0 \colon \exists t > 0 \text{ such that } \|\Delta \omega(t)\| = u\big\}, \quad \omega \in \D(\mathbb{R}^d).
		\]
		Now, suppose that \(\omega_n \to \omega\) in the local uniform topology and take some \(t > 0\). 
		As \(U(\omega)\) is at most countable, there is a \(0 < u \not \in U(\omega)\) such that \(g(x) = 0\) for \(\|x\| \leq u\). Let \(p' \triangleq \max (p \in \mathbb{Z}_+ \colon t^p (\omega, u) \leq t)\). Then, thanks to \cite[Theorem 2.6.2]{doi:10.1137/1101022}, there exists an \(N \in \mathbb{N}\) such that 
		\[
		\sum_{s \leq t} g(\Delta \omega_{n + N} (s)) = \sum_{k = 1}^{p'} g(\Delta \omega_{n + N} (t^k (\omega, u))), \quad n \in \mathbb{N}.
		\]
		As \(n \to \infty\) the r.h.s. converges to 
		\[
		\sum_{k = 1}^{p'} g(\Delta \omega (t^k (\omega, u))) = \sum_{s \leq t} g(\Delta \omega (s)).
		\]
		This completes the proof.
	\end{proof}
	By virtue of this lemma, for every \(t \in \Gamma\), we conclude that the set
	\begin{align*}
	\{(u&, \omega) \in A \colon A_u \ni \xi \mapsto Y^\circ_{t} (\xi) \text{ is discontinuous at \(\omega\)} \}
	\end{align*}
	is \(Q\)-null and consequently, that \(Y^\circ_{t}\) is \((Q_{n}, Q)\)-continuous. 
	
	Let \(\Z\) be the determining set from part (i) of Example \ref{ex: determining set} with \(D = \Gamma\). Then, it is clear that for every \(Z_s^\circ \in \Z_s\) with \(s \leq t\) the random variable \(Y^\circ_t Z_s^\circ\) is also \((Q_n, Q)\)-continuous. It remains to verify the final two parts of (A4) from Theorem \ref{theo: main approx}. We fix \(T \in \Gamma\). Let \(S_1, S_2, \dots\) be as in (ii)' and set 
	\[
	Y^n \triangleq \big( X^n (h)_{\cdot \wedge T \wedge S_n} - X^n_0 - B^n_{\cdot \wedge T \wedge S_n} \big)^{(i)} \big( X^n (h)_{\cdot \wedge T \wedge S_n} - X^n_0 - B^n_{\cdot \wedge T \wedge S_n} \big)^{(j)} - \widetilde{C}^{n, (ij)}_{\cdot \wedge T \wedge S_n},
	\]
	which is a local martingale on \(\B^n\).
	First of all, as \(|\Delta (X^n (h) - X_0 - B^n)^{(i)}| \leq 2 \|h\|_\infty\), we deduce from \cite[Lemma VII.3.34]{JS} that 
	\begin{align*}
	E^{P^n} \Big[ \sup_{s \leq {T \wedge S_n}} \big| \big( X^n (h)_{s \wedge S_n} - X^n_0 - B^n_{s \wedge S_n} \big)^{(i)} \big|^4 \Big] \lesssim E^{P^n} \Big[ \big|\widetilde{C}^{n, (ii)}_{T \wedge S_n}\big|^2\Big]^\frac{1}{2} + E^{P^n} \Big[ \big|\widetilde{C}^{n, (ii)}_{T \wedge S_n}\big|^2\Big].
	\end{align*}
	Consequently, hypothesis (ii)' yields that 
	\begin{align*}
	\sup_{n \in \mathbb{N}} E^{P^n} \Big[ \sup_{s \leq T} |Y^n_s|^2 \Big] < \infty.
	\end{align*}
	Hence, \(Y^n\) is a true martingale on \(\B^n\) and the set \(\{Y^n_s \colon s \in [0, T], n \in \mathbb{N}\}\) is uniformly integrable. It remains to verify \eqref{eq: main conv cond mart approx}.
	Note that on \(\of 0, T \wedge S_n\gs\)
	\begin{equation}\label{eq:multi bound}\begin{split}
	Y^n - Y^\circ (X^n) = \big( &X^n (h) - X^n_0 - B^n \big)^{(i)} \big( B^{n} - B (X^n) \big)^{(j)} \\&+ \big( X^n (h) - X^n_0 - B (X^n) \big)^{(j)} \big( B^{n} - B (X^n) \big)^{(i)} \\& - \widetilde{C}^{n, (ij)} + \widetilde{C}^{(ij)} (X^n).
	\end{split}\end{equation}
	Let us recall the following elementary fact (\cite[Exercise 3.5, p. 58]{kallenberg}): Let \(\xi_1, \xi_2, \dots\) and \(\eta_1, \eta_2, \dots\) be random variables such that \((\xi_n)_{n \in \mathbb{N}}\) is uniformly integrable and \(\eta_n \to 0\) in probability, then \(\xi_n \eta_n \to 0\) in probability. Using this fact and assumptions (i)' and (ii)', for all \(t \in \Gamma \cap [0, T]\) and \(\varepsilon > 0\) we obtain
	\begin{align*}
	P^n \big(| Y^n_t - Y^\circ_{t} (X^n)| \geq \varepsilon \big) &\leq P^n \big(| Y^n_t - Y^\circ_{t} (X^n)| \geq \varepsilon,t \leq S_n  \big) + P^n\big(T > S_n\big)\to 0
	\end{align*}
	as \(n \to \infty\).
	As \(T \in \Gamma\) was arbitrary and \(\Gamma \subset \mathbb{R}_+\) is dense, we conclude that (A4) holds and consequently, the claim follows.
\end{proof}
\begin{remark}
	The literature contains several conditions for tightness of processes with fixed times of discontinuity. Conditions for semimartingales are given in \cite[Theorems~VI.5.10,~IX.3.20]{JS}. We also refer to the recent article \cite{10.1214/16-ECP6} where a version of Aldous's tightness criterion for processes with fixed times of discontinuity is proved.
\end{remark}
\begin{remark}
	Hypothesis (ii)' holds for instance under the following uniform boundedness assumption: For all \(T > 0\) and \(g \in C_1 (\mathbb{R}^d)\) we have
	\[
	\sup_{\omega \in \D(\mathbb{R}^d)} \big(\| \widetilde{C}_T (\omega)\| + | (g * \nu_T)(\omega) | \big) < \infty.
	\]
	This follows from arguments used in the proof of Theorem \ref{theo: JS generalization}. 
	In practice (ii)' seems to be more flexible than this boundedness condition. For instance, the assumption also holds in case
	\[
	\|\widetilde{C}^n\| \lesssim 1 + \sup_{s\leq \cdot} \|X^n_s\|^2 
	\]
	and 
	\[
	\sup_{n \in \mathbb{N}} E^{P^n}\Big[ \sup_{s \leq T} \|X^n_s\|^4 \Big] < \infty, \quad T > 0.
	\]
	Under suitable linear growth assumptions on the characteristics \((B^n, C^n, \nu^n)\) the fourth moment condition can be verified by Gronwall's lemma. 
\end{remark}

In the following two examples we explain how the control variables \(L^1, L^2, \dots\) can be chosen such that \eqref{eq: all in A^n assp2} holds when \(X^1, X^2, \dots\) are stochastic integrals. In Sections \ref{sec: annealed smg FTD} and \ref{sec: ito FTD} below we specify our setting further such that \eqref{eq: all in A^n assp2} can be interpreted more easily.

\begin{example} \label{sec: semimartingale controls}
	In this example we explain how (iii)' can be verified in case \(X^n\) is a stochastic integral, where we borrow ideas from \cite{doi:10.1080/17442508108833169,JM81}.
	As we only want to fix ideas, suppose that all semimartingales \(X^1, X^2, \dots\) are one-dimensional, defined on the same
	stochastic basis \(\B = (\Omega, \mathcal{F}, \F, P)\) and are stochastic integrals of the form
	\[
	d X^n_t = \sigma^n_t d Z^n_t,
	\]
	for a one-dimensional semimartingale \(Z^n\) and a predictable process \(\sigma^n \in L (Z^n)\).
	In this case we have 
	\(
	\Delta X^n = \sigma^n \Delta Z^n.
	\)
	Assume that there exists a non-negative predictable process \(\gamma\) such that \(\gamma \in L(Z^n)\) and \(|\sigma^n| \leq \gamma (1 + \sup_{s \leq \cdot} |X^n_s|)\) for all \(n \in \mathbb{N}\). Here, the linear growth condition can be relaxed. Then, 
	\[
	|\Delta X^n| \leq \gamma |\Delta Z^n| \Big(1 + \sup_{s \leq \cdot} |X^n_s|\Big) = |\Delta L^n| \Big(1 + \sup_{s \leq \cdot} |X^n_s|\Big), \qquad L^n \triangleq \int_0^\cdot \gamma_s d Z^n_s.
	\]
	Now, we set \(U \triangleq \D(\mathbb{R})\) and \(\u (u, t) \triangleq |\Delta u (t)|\) for \((u, t) \in \D(\mathbb{R}) \times \mathbb{R}_+\). By standard properties of \cadlag functions, the set \(\{t \in [0, T] \colon \u(u,t) \geq a\} = \{t \in [0, T] \colon |\Delta u(t)| \geq a\}\) is finite for every \((a, T, u) \in (0, \infty) \times (0, \infty) \times \D(\mathbb{R})\). Consequently, \eqref{eq: A} holds by Lemma \ref{lem: set G finite jump}. Often enough, one has \(Z^n \equiv Z\), which further implies that the law of \(L^{n}\) is independent of \(n\). 
	
	Alternatively, suppose that there exists a \cadlag measurable process \(Z\) such that \(Z^n \to Z\) in the ucp\footnote{ucp = uniformly on compacts in probability} topology, i.e. for all \(t \in \mathbb{R}_+\)
	\[
	\sup_{s \leq t} |Z^n_s - Z_s| \to 0
	\]
	in probability as \(n \to \infty\), and fix \(T > 0\). Then, up to passing to a subsequence which we ignore for simplicity, the set 
	\[
	\Omega^o \triangleq \Big\{ \omega \in \Omega \colon \lim_{n \to \infty} \sup_{s \leq T} |Z^n_s (\omega) - Z_s (\omega)| = 0 \Big\}
	\]
	is full. Now, when we define
	\[
	\u (\omega, t) \triangleq \begin{cases} \sup_{n \in \mathbb{N}} |\Delta Z^n_t (\omega)|,& \omega \in \Omega^o,\\
	0,&\text{otherwise},\end{cases}
	\]
	for \((\omega, t) \in \Omega \times [0, T]\), the set \(\{t \in [0, T] \colon \u (\omega, t) \geq a\}\) is finite for every \(\omega \in \Omega\) and \(a > 0\). 
	
	To see this, take \(\omega \in \Omega^o\) and let \(N = N(\omega) \in \mathbb{N}\) be such that 
	\[
	\sup_{n \geq N}\sup_{s \leq T} | Z^n_s (\omega) - Z_s (\omega) | \leq \frac{a}{3}.
	\]
	Then, for every \(t \in [0, T]\)
	\begin{align*}
	\sup_{n \geq N} |\Delta Z^n_t (\omega)| \geq a \ \Longrightarrow \ | \Delta Z_t (\omega)| &\geq \sup_{n \geq N}|\Delta Z^n_t (\omega)| - \sup_{n \geq N} | \Delta Z^n_t (\omega) - \Delta Z_t (\omega) | \\&\geq a - \frac{2a}{3} = \frac{a}{3}.
	\end{align*}
	As there are only finitely many \(t \in [0, T]\) such that \(|\Delta Z_t (\omega)| \geq a/3\), there are also only finitely many \(t \in [0, T]\) such that \(\sup_{n \geq N} |\Delta Z^n_t (\omega)| \geq a\). 
	Now, since
	\[
	\Big\{t  \colon \sup_{n \in \mathbb{N}} |\Delta Z^n_t (\omega)| \geq a \Big\} \subset 
	\Big(\bigcup_{k = 1}^{N - 1} \Big\{ t  \colon |\Delta Z^k_t (\omega)| \geq \frac{a}{2}\Big\}\Big) \cup \Big\{t \colon \sup_{n \geq N} |\Delta Z^n_t (\omega)| \geq \frac{a}{2}\Big\},
	\]
	we conclude that there are at most finitely many \(t \in [0, T]\) such that \(\sup_{n \in \mathbb{N}} |\Delta Z^n_t (\omega)| \geq a\), which was the claim.
	
	Up to a pasting argument, if \(\gamma\) is e.g. constant, we can take \(U \triangleq \Omega\) and \(L^{n}\equiv \on{Id}\) such that \eqref{eq: A} holds.
	In particular, as we assume that all processes are defined on the same stochastic basis, also (a) in Theorem \ref{theo: gen JS local uniform topology 2} holds. At the cost of slightly more complicated conditions related inter alia to part (b) of Theorem \ref{theo: gen JS local uniform topology 2}, this argument can be transferred to the more general case where \(\B^n = (\Omega, \mathcal{F}, \F^n, P^n)\). More details on this strategy are given in the proof of Corollary \ref{coro: main coro IP} below.
\end{example}
\begin{example}\label{sec: RM}
	In this example we explain how (iii)' can be checked in case \(X^1, X^2, \dots\) are stochastic integrals w.r.t. a (compensated) random measure. 
	As in Example \ref{sec: semimartingale controls}, for simplicity assume that all \(X^1, X^2, \dots\) are one-dimensional and defined on the same stochastic basis. Moreover, we assume that
	\[
	X^n = X^n_0 +  \int_0^\cdot \int H^n(s, y) (\p^n - \q^n) (ds, dy),
	\]
	where \(\p^n - \q^n\) is a compensated integer-valued random measure on a Blackwell space \((E, \mathcal{E})\) and \(H^n \in G_\textup{loc} (\p^n)\).
	Suppose that \(\gamma\) is a non-negative predictable process such that a.s. \(\gamma * \q^n < \infty\) and  \(|H^n| \leq \gamma (1 + \sup_{s \leq \cdot} |X^n_s|)\) for all \(n \in \mathbb{N}\). Here, the linear growth condition can be relaxed.
	We now set 
	\[
	L^{n} \triangleq \int_0^\cdot \int \gamma (s, y) (\p^n + \q^n) (ds, dy),
	\]
	and we obtain that
	\begin{align*}
	|\Delta X^n_t| &= \Big| \int H^n (t, y) \p^n(\{t\} \times dy) -  \int H^n (t, y) \q^n (\{t\} \times dy) \Big|
	\\&\leq \int | H^n(t, y) | \p^n(\{t\} \times dy) + \int | H^n(t, y) | \q^n(\{t\} \times dy)
	\\&\leq \int \gamma (t, y) (\p^n + \q^n) (\{t\} \times dy) \Big(1 + \sup_{s \leq t} |X^n_s|\Big)  
	\\&= | \Delta L^{n}_t | \Big(1 + \sup_{s \leq t} |X^n_s|\Big) 
	\end{align*}
	for all \(t > 0\). Now, we can define \(U = \D(\mathbb{R})\) and \(\u (u, t) = |\Delta u(t)|\) such that \eqref{eq: A} holds. Often enough the law of \(L^{n}\) is furthermore independent of \(n\). The strategy outlined in the second part of Example \ref{sec: semimartingale controls} can also be transferred into this setting, see the proof of Corollary~\ref{coro: main coro IP} below.
\end{example}

\subsubsection{The Annealed Setting} \label{sec: annealed smg FTD}
In this section we assume that \(X^1, X^2, \dots\) are defined on the same filtered probability space, which can be viewed as an annealed setting, see Section \ref{sec: motivation}. In this case we allow the limiting characteristics to be random. 

We fix a filtered probability space \(\B = (\Omega, \mathcal{F}, \F, P)\) which supports \(\mathbb{R}^d\)-valued c\'adl\'ag adapted processes \(X^1, X^2, \dots\). 
Moreover, we define an \emph{extension} \((\Omega', \mathcal{F}', \F')\) of the filtered space \((\Omega, \mathcal{F}, \F)\) by 
\begin{align*}
\Omega' \triangleq \Omega \times \D(\mathbb{R}^d), \quad \mathcal{F}' \triangleq \mathcal{F} \otimes \mathcal{D}(\mathbb{R}^d), \quad \mathcal{F}'_t \triangleq \bigcap_{s > t} \big(\mathcal{F}_s \otimes \mathcal{D}_s (\mathbb{R}^d)\big), 
\end{align*}
where \(\mathcal{D}(\mathbb{R}^d)\) and \((\mathcal{D}_t(\mathbb{R}^d))_{t \geq 0}\) are the canonical \(\sigma\)-field and the canonical (right-continuous) filtration on \(\D(\mathbb{R}^d)\). We define a \emph{canonical process} on \(\Omega'\) by \(\X(\omega, \alpha) = \alpha\) for~\((\omega, \alpha) \in \Omega'\).

Let \((B, C, \nu)\) be a candidate triplet on \((\Omega', \mathcal{F}', \F')\) relative to a fixed continuous truncation function \(h \colon \mathbb{R}^d \to \mathbb{R}^d\), cf. \cite[III.2.3]{JS}. 
Let \(\u \colon \Omega \times \mathbb{R}_+ \to \mathbb{R}_+\) be an \(\mathcal{F} \otimes \mathcal{B}(\mathbb{R}_+)/\mathcal{B}(\mathbb{R}_+)\) measurable function such that for every \((\omega, t) \in \Omega \times (0, \infty)\)
\begin{align*}
\lim_{\varepsilon \searrow 0} \sup \big\{ \u(\omega, s) \colon s \not = t, t - \varepsilon \leq s \leq t + \varepsilon\big\} = 0,
\end{align*}
let \(\kappa \colon \mathbb{R}_+ \to \mathbb{R}_+\) be increasing and continuous, and define 
\begin{align*} 
A^{n} &\triangleq \Big\{ \omega \in \Omega \colon \|\Delta X^{n}_t (\omega)\| \leq \u (\omega, t) \kappa \Big(\sup_{s\leq t}\|X^n_s(\omega)\|\Big) \text{ for all } t > 0 \Big\}.
\end{align*}
The set \(\bigcap_{n \in\mathbb{N}} A^n\) can be interpreted as follows: The jumps of the processes \(X^1, X^2, \dots\) are controlled by a process \(\u\) which roughly behaves like the jump process \(|\Delta Z|\) of some one-dimensional \cadlag process \(Z\). 

The main result of this section is the following:
\begin{theorem} \label{theo: gen JS local uniform topology 3}
	Suppose that each \(X^n\) is a semimartingale with semimartingale characteristics \((B^n, C^n,\) \(\nu^n)\) corresponding to the (continuous) truncation function \(h\). 
	Assume that there exists a probability measure \(P \circ X^{-1}\) on \((\D(\mathbb{R}^d), \mathcal{D}(\mathbb{R}^d))\) such that \(P \circ (X^n)^{-1} \to P \circ X^{-1}\) weakly (where \(\D(\mathbb{R}^d)\) is endowed with the Skorokhod \(J_1\) topology) and that there exists a dense set \(\Gamma \subset \mathbb{R}_+\) such that the following hold:
	\begin{enumerate}
		\item[\textup{(i)}] For every \(t \in \Gamma, \varepsilon > 0\) and \(g \in C_1(\mathbb{R}^d)\) we have 
		\begin{align*}
		P (\|B^n_{t}  - B_{t} (X^n)\| \geq \varepsilon) &\to 0,\\ P (\|\widetilde{C}^n_{t}  - \widetilde{C}_{t} (X^n)\| \geq \varepsilon) &\to 0,\\ P(|g * \nu^n_{t}  - g * \nu_{t} (X^n)| \geq \varepsilon) &\to 0,
		\end{align*}
		as \(n \to \infty.\)
		\item[\textup{(ii)}]
		For all \(T \in \Gamma\) and \(g \in C_1(\mathbb{R}^d)\) there is a sequence of stopping times \((S_n)_{n \in \mathbb{N}}\) such that
		\[
		P (S_n < T) \to 0, \quad n \to \infty,
		\]
		and
		\begin{align*}
		\sup_{n \in \mathbb{N}} E^{P} \Big[ \|\widetilde{C}^n_{T \wedge S_n}\|^2 + g^2 * \nu^n_{T \wedge S_n} \Big] < \infty.
		\end{align*}
		\item[\textup{(iii)}]	For all \(\omega \in \Omega, t \in \Gamma\) and \(f \in C_1(\mathbb{R}^d)\) the maps 
		\[
		\D(\mathbb{R}^d) \ni \alpha  \mapsto B_t (\omega, \alpha), \widetilde{C}_t (\omega, \alpha), (f * \nu_t) (\omega, \alpha)
		\]
		are continuous in the local uniform topology. Moreover,
		\begin{align*}
		P (A^{n}) \to 1\text{ as } n \to \infty.
		\end{align*}
	\end{enumerate}
	Then, there exists a probability measure \(Q\) on \((\Omega', \mathcal{F}')\), which is a weak-strong accumulation point of \(\{P \circ (\on{Id}, X^n)^{-1}\colon n \in \mathbb{N}\}\), with \(Q_\Omega = P\) and \(Q_{\D(\mathbb{R}^d)} = P \circ X^{-1}\) such that on \((\Omega', \mathcal{F}', \F', Q)\) the canonical process \(\X\) is a semimartingale with characteristics \((B, C, \nu)\). 
\end{theorem}
\begin{proof}
	The proof is similar to those of Theorem \ref{theo: gen JS local uniform topology 2} where we take \((U, \mathcal{U}) = (\Omega, \mathcal{F})\) and \(L^{n} (\omega) \equiv L (\omega) = \omega\) for all \(\omega \in \Omega\). 
	The details are left to the reader.
\end{proof}
In the context of SDEs with semimartingale drivers, Theorem \ref{theo: gen JS local uniform topology 3} is related to \cite[Theorem~3.16]{JM81}.

\begin{example}
	We provide a short example for an application of Theorem \ref{theo: gen JS local uniform topology 3} in a setting without jumps. A more detailed exposition of a closely related setting with fixed times of discontinuity is given in Section \ref{sec: ito FTD} below. 
	We take \(\B\) as underlying filtered space. Let \(\tau^0, \tau^1, \tau^2, \dots\) be stopping times on \(\B\), which we think to be change points of economic scenarios, and let \(W\) be a one-dimensional standard Brownian motion on the stochastic basis \(\B\). Moreover, take \(b, b^\circ \colon \mathbb{R}_+ \times \mathbb{R} \to \mathbb{R}\) and \(\sigma, \sigma^\circ \colon \mathbb{R}_+ \times \mathbb{R} \to \mathbb{R}\) to be sufficiently regular functions such that for each \(n \in \mathbb{Z}_+\) the SDE
	\begin{align*}
	d X^n_t = \big( b (t, X^n_t) \1_{\{t \leq \tau^n\}} &+ b^\circ(t, X^n_t) \1_{\{t > \tau^n\}}\big) dt \\&+ \big( \sigma (t, X^n_t) \1_{\{t \leq \tau^n\}} + \sigma^\circ(t, X^n_t) \1_{\{t > \tau^n\}}\big) d W_t, \quad X^n_0 = x_0, 
	\end{align*}
	has a solution process \(X^n\). It is well-known that (local) Lipschitz (or monotonicity) and linear growth conditions on \(b, b^\circ\) and \(\sigma, \sigma^\circ\) imply existence (and uniqueness in a strong sense), see, e.g. \cite[Chapter 14]{J79} or \cite[Section 4]{JM81}. We also stress that the above SDEs have random coefficients, where the randomness enters in terms of the sequence \(\tau^0, \tau^1, \tau^2, \dots\). 
	Part (i) of Theorem \ref{theo: gen JS local uniform topology 3} holds if \(\tau^n \to \tau^0\) in probability. Under this condition and under suitable assumptions on the coefficients (see \cite[Section~4]{JM81}), if \(X^1, X^2, \dots\) converge weakly, then\footnote{More precisely, under (local) Lipschitz or monotonicity conditions as given in \cite[Section 4]{JM81}, by \cite[Corollary~2.26]{JM81} there exists a unique solution measure (in the sense of \cite[Definition 1.6]{JM81}) to the SDE of \(X^0\) and it is strong (see \cite[Definition 2.21]{JM81}). Thus, by \cite[Theorem 2.22]{JM81}, the solution measure has the form \(\delta_{X^0 (\omega)}(d \alpha) P(d \omega)\). As \(Q\) in Theorem~\ref{theo: gen JS local uniform topology 3} is a solution measure to the SDE of \(X^0\) by \cite[Theorem 2.10]{JM81}, the claim follows.} the measure \(Q\) in Theorem \ref{theo: gen JS local uniform topology 3} is given by 
	\[
	Q (d \omega, d \alpha) = \delta_{X^0 (\omega)} (d \alpha) P(d \omega), 
	\]
	which, by virtue of \(Q_{\D(\mathbb{R})} = P \circ (X^0)^{-1}\), yields that the laws of \(X^1, X^2, \dots\) converge weakly to the law of \(X^0\). In fact, we can say more: By Remark \ref{rem: ws conv, in prob}, we can even conclude that \(X^n \to X^0\) in the ucp topology. This observation can be compared to classical ucp stability results for semimartingale SDEs as for instance given in \cite{protter78}. We stress that the above argument does not rely on the Lipschitz continuity of the coefficients as the argument in \cite{protter78}, but on strong existence and uniqueness. Therefore, we think it is more flexible when it comes to  the regularity of the coefficients. However, in the presence of jumps the argument only yields converges in probability for the Skorokhod \(J_1\) topology, which is weaker than ucp convergence.
\end{example}

\subsubsection{Application: It\^o Processes with Fixed Times of Discontinuity} \label{sec: ito FTD}
In this section we specify Theorem \ref{theo: gen JS local uniform topology 2} for solutions to SDEs driven by a Gaussian continuous local martingale and a Poisson random measure. 

Let \((E, \mathcal{E})\) be a Blackwell space and let \(\mathcal{P}\) be the predictable \(\sigma\)-field on \(\mathbb{R}_+ \times \D(\mathbb{R}^d)\) when \(\D(\mathbb{R}^d)\) is equipped with the canonical filtration.
Let \(\sigma \colon \mathbb{R}_+ \times \D(\mathbb{R}^d) \to \mathbb{R}^{d \times r}\) be \(\mathcal{P}/\mathcal{B}(\mathbb{R}^{d \times r})\) measurable and let \(v, b \colon \mathbb{R}_+ \times \D(\mathbb{R}_+) \times E \to \mathbb{R}^d\) be \(\mathcal{P} \otimes\mathcal{E}/\mathcal{B}(\mathbb{R}^d)\) measurable. 
Let \(C^n\) and \(C\) be covariance functions for \(r\)-dimensional continuous Gaussian martingales such that 
\[
C^n = \int_0^\cdot c^n_s ds, \qquad C = \int_0^\cdot c_s ds.
\]
Moreover, let \(\q\) and \(\q^n\) be intensity measures of Poisson random measures on \((E, \mathcal{E})\), let \(q^n\) and \(q\) be \(\sigma\)-finite measures on \((\mathbb{R}_+ \times E, \mathcal{B}(\mathbb{R}_+) \otimes \mathcal{E})\) and let \(h \colon \mathbb{R}^d \to \mathbb{R}^d\) be a continuous truncation function.
For each \(n \in \mathbb{N}\) we fix a stochastic basis \(\B^n \triangleq (\Omega, \mathcal{F}, \F^n, P)\) which supports the following: A continuous  Gaussian martingale \(W^n\) with covariance function~\(C^n\), a Poisson random measure \(\p^n\) with intensity measure \(\q^n\) and a semimartingale \(X ^n\) with dynamics
\begin{align*}
X^n = X^n_0 + \int_0^\cdot \int b^n (t, y) q^n(dt, dy) &+ \int_0^\cdot \sigma^n_t d W^n_t + \int_0^\cdot \int h(v^n(t, y)) (\p^n - \q^n)(dt, dy) \\& + \int_0^\cdot \int (v^n(t, y) - h(v^n(t, y))) \p^n (dt, dy),
\end{align*}
where \(b^n, \sigma^n\) and \(v^n\) are suitable processes such that the integrals are well-defined.
We now formulate some technical conditions:
\begin{listing} \label{ass: A1 IP}
	There exists a \cadlag process \(X\) on some probability space \((\Omega^*, \mathcal{F}^*, P^*)\) such that \(X^n \to X\) weakly on \(\D(\mathbb{R}^d)\) endowed with the Skorokhod \(J_1\) topology. Let \(\F^X\) be the canonical (right-continuous) filtration generated by \(X\) and set \(\B \triangleq (\Omega^*, \mathcal{F}^*, \F^X, P^*)\).
\end{listing}

\begin{listing} \label{ass: A2 IP}
	Part (i)' and (ii)' of Theorem \ref{theo: gen JS local uniform topology 2} hold for a dense set \(\Gamma \subset \mathbb{R}_+\) and the following characteristics \((B^n, C^n, \nu^n)\) and \((B, C, \nu)\):
	\begin{align*}
	B^n &\triangleq \int_0^\cdot \int b^n (t, y) q^n(dt, dy), \\
	C^n &\triangleq \int_0^\cdot \sigma^n_t c^n_t \sigma^{n*}_t dt,\\
	\nu^n([0, t] \times G) &\triangleq \int_0^t \int \1_G (v^n(s, y)) \q^n (ds, dy), \quad t \in \mathbb{R}_+, G \in \mathcal{B}(\mathbb{R}^d\backslash \{0\}),
	\end{align*}
	and 
	\begin{align*}
	B &\triangleq \int_0^\cdot \int b (t, y) q(dt, dy), \\
	C &\triangleq \int_0^\cdot \sigma_t c_t \sigma^*_t dt,\\
	\nu([0, t] \times G) &\triangleq \int_0^t \int \1_G (v(s, y)) \q (ds, dy), \quad t \in \mathbb{R}_+, G \in \mathcal{B}(\mathbb{R}^d\backslash \{0\}).
	\end{align*}
	It is implicit\footnote{In particular, \(\Delta B^n_t = \int h(x) \nu^n(\{t\} \times dx)\), which means \(\int b^n(t, y) q^n(\{t\} \times dy) = \int h (v^n(t, y)) \q^n(\{t\} \times dy)\). As a consequence, \(\Delta X^n_t = \int v^n(t, y) \p^n(\{t\} \times dy)\).} that \((B^n, C^n, \nu^n)\) and \((B, C, \nu)\) are candidate triplets in the sense of \cite[III.2.3]{JS}.
	Moreover, for every \(t \in \Gamma\) and \(g \in C_1 (\mathbb{R}^d)\) the functions \(B_t, \widetilde{C}_t\) and \(g * \nu_t\) are continuous on \(\mathbb{D}(\mathbb{R}^d)\) endowed with the local uniform topology.
\end{listing}

\begin{listing} \label{ass: A3 IP}
	Let \(\k\colon \mathbb{R}_+\to\mathbb{R}_+\) be increasing and continuous. For each \(T \in \mathbb{N}\) there exists a sequence \(\gamma^n = \gamma^{n, T} \colon \mathbb{R}_+ \times \Omega \times E \to \mathbb{R}_+\) of non-negative \(\mathcal{B}(\mathbb{R}_+)\otimes \mathcal{F} \otimes \mathcal{E}/\mathcal{B}(\mathbb{R}_+)\) measurable functions with the following properties: 
	\begin{enumerate}
		\item[\textup{(i)}] a.s. \(\gamma^n * \p^n_T < \infty\) for all \(n \in \mathbb{N}\).
		\item[\textup{(ii)}] \(\|v^n\| \leq \gamma^n \k (\sup_{s \leq \cdot} \|X^n_s\|)\) on \(\Omega \hspace{0.05cm}\times(T - 1, T] \times E\).
		\item[\textup{(iii)}] There exists a \cadlag measurable process \(Z = Z^{T}\) such that 
		\[
		\sup_{s \leq T} |\gamma^n * \p^n_s - Z_s| \to 0 
		\]
		in probability \(n \to \infty\).
	\end{enumerate}
\end{listing}

\begin{corollary} \label{coro: main coro IP}
	Suppose that Assumptions \ref{ass: A1 IP}, \ref{ass: A2 IP} and \ref{ass: A3 IP} hold. Then, \(X\) is a semimartingale on~\(\B\) whose semimartingale characteristics are given by \((B(X), C(X), \nu(X))\). Possibly on a standard extension of \(\B\), there is a Gaussian continuous martingale \(W\) with covariance function \(C\) and a Poisson random measure \(\p\) with intensity measure \(\q\) such that 
	\begin{align*}
	X = X_0 + \int_0^\cdot \int b (t, X, y) q(dt, dy) &+ \int_0^\cdot \sigma_t (X) d W_t + \int_0^\cdot \int h(v(t, X, y)) (\p - \q)(dt, dy) \\& + \int_0^\cdot \int (v(t,X, y) - h(v(t,X, y))) \p (dt, dy).
	\end{align*}
\end{corollary}

\begin{proof}
	Our strategy is to apply Theorem \ref{theo: gen JS local uniform topology 2}. 
	By hypothesis, referring to Theorem \ref{theo: gen JS local uniform topology 2}, (i)' and (ii)' and the continuity assumption from (iii)' hold. 
	Via passing to a subsequence, which we ignore in our notation for simplicity, we can assume that a.s. for all \(T \in \mathbb{N}\)
	\[
	\sup_{s \leq T} | \gamma^{n, T} * \p^n_s - Z^{T}_s | \to 0
	\]
	as \(n \to \infty\). 
	Define 
	\[
	\Omega^o \triangleq \Big\{ \omega \in \Omega \colon \lim_{n \to \infty} \sup_{t \leq T} | \gamma^{n, T} * \p^n_t - Z^{T}_t | = 0, \quad T = 1, 2, \dots \Big\},
	\]
	which then is a full set.
	Let \((U, \mathcal{U}) = (\Omega, \mathcal{F})\) and \(L^{n} (\omega) = \omega\) be the corresponding identity map. Clearly, for these choices part (a) in Theorem \ref{theo: gen JS local uniform topology 2} holds. 
	For \(T \in \mathbb{N}\) and \((\omega, t) \in \Omega \times (T - 1, T]\) we define 
	\[
	\u (\omega, t) \triangleq \begin{cases} \sup_{n \in \mathbb{N}} \int \gamma^{n, T} (\omega; t, y) \p^n(\omega; \{t\} \times dy), & \omega \in \Omega^o,\\ 0,&\text{otherwise}.\end{cases}
	\]
	We also set \(\u (\omega, 0) \equiv 0\) for all \(\omega \in \Omega\).
	By definition of \(\Omega^o\), the fact that \(Z^{T}\) has \cadlag paths and standard properties of \cadlag functions, for every \(\omega \in \Omega\) and \(a > 0\) the set
	\(
	\{t \in (T- 1, T] \colon \u(\omega, t) \geq a \}\)
	is finite, see Example~\ref{sec: semimartingale controls} for more details. Consequently, by Lemma \ref{lem: set G finite jump}, \(\u\) is as in Section \ref{sec: FTD main JS gen} and it remains to verify \eqref{eq: A}.
	For every \(T - 1 < t \leq T\) we have a.s.
	\begin{align*}
	\|\Delta X^n_t\| &\leq \int \|v^n(t, y)\| \p^n(\{t\} \times dy)\\& \leq  \int \gamma^{n, T} (t, y) \p^n (\{t\} \times dy) \k \Big(\sup_{s \leq t} \|X^n_s\|\Big) \\&\leq \u (L^{n}, t) \k \Big(\sup_{s \leq t} \|X^n_s\|\Big) ,
	\end{align*}
	which shows \eqref{eq: A}. In summary, we conclude from Theorem \ref{theo: gen JS local uniform topology 2} that \(X\) is a semimartingale (for its canonical filtration) with characteristics \((B(X), C(X), \nu(X))\). 
	
	The final claim, i.e. the representation of \(X\) as stochastic integrals, follows from classical representation theorems as given in \cite{KLS} and \cite[Section XIV.3]{J79}.
\end{proof}

There is also a version of Corollary \ref{coro: main coro IP} for the case \(\B^n = (\Omega, \mathcal{F}, \F^n, P^n)\), i.e. with varying probability measures. Before we present this version, let us emphasis that even if the probability measures on \(\B^n\) are allowed to be different, we ask them to be quite close in the sense that they converge to each other in \emph{total variation}. 

\begin{corollary}
	Corollary \ref{coro: main coro IP} holds for \(\B^n = (\Omega, \mathcal{F}, \F^n, P^n)\)  and \(\B = (\Omega, \mathcal{F}, \F^X, P)\) under the following assumptions:
	The \(\sigma\)-field \(\mathcal{F}\) is separable, 
	\begin{align}\label{eq: very strong conv}
	\sup_{G \in \mathcal{F}} | P^n(G) - P(G) | \to 0
	\end{align}
	as \(n \to \infty\), Assumptions \ref{ass: A1 IP} and \ref{ass: A2 IP} hold, and Assumption \ref{ass: A3 IP} holds with \textup{(i)} and \textup{(iii)} replaced by 
	\begin{enumerate}
		\item[\textup{(i)'}] \(P^n\)-a.s. \(\gamma^n * \p^n_T < \infty\) for all \(n \in \mathbb{N}\).
		\item[\textup{(iii)'}] There exists a \cadlag measurable process \(Z = Z^{T}\) such that for all \(\varepsilon > 0\)
		\[
		P^n \Big( \sup_{s \leq T}| \gamma^n * \p^n_s - Z_s| \geq \varepsilon \Big) \to 0
		\]
		as \(n \to \infty\).
	\end{enumerate}
\end{corollary}
\begin{proof}
	As \(\mathcal{F}\) is separable, \(M_m (\Omega)\) is metrizible thanks to Theorem \ref{theo: Mmc metrizible}. By virtue of Remark \ref{rem: topology M_m(U)}, \eqref{eq: very strong conv} implies \(P^n \to P\) in \(M_m(\Omega)\). Thus, \(\{P^n \colon n \in \mathbb{N}\}\) is relatively compact in \(M_m(\Omega)\). In other words, for \((U, \mathcal{U}) = (\Omega, \mathcal{F})\) and \(L^{n} = \textup{Id}\), part (b) in Theorem \ref{theo: gen JS local uniform topology 2} holds. Under \eqref{eq: very strong conv} we have for every \(\varepsilon > 0\)
	\[
	P^n \Big( \sup_{s \leq T}| \gamma^n * \p^n_s - Z_s| \geq \varepsilon\Big) \to 0 \quad \Longleftrightarrow \quad P \Big( \sup_{s \leq T}| \gamma^n * \p^n_s - Z_s| \geq \varepsilon \Big) \to 0.
	\]
	Furthermore, with \(\Omega^o\) as in the proof of Corollary \ref{coro: main coro IP}, 
	\(
	P^n(\Omega^o) \to P(\Omega^o) = 1.
	\)
	With these observations at hand, the proof of Corollary \ref{coro: main coro IP} needs no further change.
\end{proof}

Finally, we remark that Theorem \ref{theo: gen JS local uniform topology 3} can also be transferred to the current setting. This yields a stability result for random coefficients \(b, \sigma\) and \(v\). We leave the precise statement to the reader.

\end{document}